\documentclass[12pt, oneside, reqno]{amsart}
\usepackage{amsmath,amssymb,latexsym,soul,cite,mathrsfs}
\usepackage{color,enumitem,graphicx}
\usepackage[colorlinks=true,urlcolor=blue,
citecolor=red,linkcolor=blue,linktocpage,pdfpagelabels,
bookmarksnumbered,bookmarksopen]{hyperref}
\usepackage[english]{babel}
\usepackage[left=2.7cm,right=2.7cm,top=3.2cm,bottom=3.2cm]{geometry}

\usepackage[hyperpageref]{backref}
\numberwithin{equation}{section}
\newtheorem{theorem}{Theorem}[section]
\newtheorem{lemma}{Lemma}[section]

\newtheorem{proposition}{Proposition}[section]

\newtheorem{remark}{Remark}[section]

\title[Elliptic problems on Riemannian manifolds]
 {Some elliptic problems with singular nonlinearity and advection for Riemannian manifolds}

\author[J.M.\ do \'O]{Jo\~ao Marcos do \'O}
\author[R.G.\ Clemente]{Rodrigo G.\ Clemente}

\address[J.M. do \'O]{Department of Mathematics,
Federal University of Para\'{\i}ba
\newline\indent
58051-900, Jo\~ao Pessoa-PB, Brazil}
\email{\href{mailto:jmbo@pq.cnpq.br}{jmbo@pq.cnpq.br}}

\address[R.\ Clemente]{Department of Mathematics, 
Rural Federal University of Pernambuco
\newline\indent 
52171-900, Recife, Pernambuco, Brazil}
\email{\href{mailto:rodrigo.clemente@ufrpe.br}{rodrigo.clemente@ufrpe.br}}

\thanks{Research partially supported by the National Institute of Science and Technology of Mathematics INCT-Mat, CAPES and CNPq.}

\usepackage[norefs,nocites]{refcheck}
\begin{document}

\begin{abstract}
We are interested in regularity properties of semi-stable solutions for a class of singular semilinear elliptic problems with advection term defined on a smooth bounded domain  of a complete Riemannian manifold with zero Dirichlet boundary condition. We prove uniform Lebesgue estimates and we determine the critical dimensions for these problems with nonlinearities of the type Gelfand, MEMS and power case. As an application, we show that extremal solutions are classical whenever the dimension of the manifold is below the critical dimension of the associated problem. Moreover, we analyze the branch of minimal solutions and we prove multiplicity results when the parameter is close to critical threshold and we obtain uniqueness on it. Furthermore, for the case of Riemannian models we study properties of radial symmetry and monotonicity for semi-stable solutions.
\end{abstract}

\maketitle

\section{Introduction}
Let $(\mathcal M, g)$ be a complete Riemannian manifold with dimension $N$, $\Omega\subset\mathcal{M}$ a smooth bounded domain and $A(x)$ a smooth vector field over $\overline{\Omega}$. In the present paper, we investigate the following  class of nonlinear elliptic differential equations involving singular nonlinearities and advection

\begin{equation}\label{01}
\left\{
\begin{alignedat}{3}
-\Delta_g u + A(x)\cdot\nabla_g u= & \, \lambda f(u) & \quad \text{in} & \quad\Omega, \\
u > &\, 0 & \quad \text{in} & \quad \Omega,\\
u = & \, 0  &  \text{on} & \quad \partial\Omega,\\
\end{alignedat}
\right.\tag{$P_{\lambda}$}
\end{equation}
We analyse \eqref{01} for the following types of  nonlinearities:

\begin{flalign}\label{60}
\begin{array}{lccclll}
(i)   &  & f(s) & = & e^s                 &  & \hspace{2cm}\left(\text{Gelfand}\right)    \vspace{0.15cm}\\ 
(ii)  &  & f(s) & = & (1+s)^m\text{, }m>1 &  & \hspace{2cm}\left(\text{Power-type}\right) \vspace{0.15cm}\\
(iii) &  & f(s) & = & 1/(1-s)^2           &  & \hspace{2cm}\left(\text{MEMS}\right)       \vspace{0.15cm}
\end{array}&&
\end{flalign}

The main purpose of this paper is to study the minimal branch and regularity properties for minimal solutions of \eqref{01}. We first prove that there exists some positive finite critical paramater $\lambda^*$ such that for all $0<\lambda < \lambda^*$ the problem \eqref{01} has a smooth minimal stable solution $\underline{u}_{\lambda}$ while for $\lambda > \lambda^*$ there are no solutions of \eqref{01} in any sense (cf.  Theorems \ref{22}). We determine the critical dimension $N^*$ for this class  of problems, precisely we prove that the extremal solution of \eqref{01} is regular for $N \leq N^*$ and it is singular for $N > N^*$. We see that the critical dimension depends only on the nonlinearity $f(s)$ and does not depend of the Riemanian manifold $\mathcal{M}$ (cf. Theorem \ref{13} and Table \ref{64}). For that, we establish $L^{\infty}$ estimates, which are crucial in our argument to obtain regularity of the extremal solutions. We also prove multiplicity of solutions near the extremal parameter and uniqueness on it (cf. Theorem \ref{36} and Theorem \ref{37}). Moreover, we prove radial symmetry and monotonicity for semi-stable solutions of \eqref{01} if $\Omega=\mathcal{B}_R$ is a geodesic ball of a Riemannian model $\mathcal{M}$ (cf. Theorem \ref{07}). 

\subsection{Statement of main results} Before we state our main results we recall some standard notations and definitions related with problem \eqref{01}. Next we are assuming the following values for $s_0$, which depends of the type of considered nonlinearity, precisely,
\begin{flalign*}
\begin{array}{lccclll}
(i)   &   s_0 & = & +\infty & \text{if} & f(s) = e^s       & \left(\text{Gelfand}\right)\vspace{0.15cm}\\ 
(ii)  &   s_0 & = & +\infty & \text{if} & f(s) = (1+s)^m   & \left(\text{Power-type}\right)\vspace{0.15cm}\\
(iii) &   s_0 & = & 1       & \text{if} & f(s) = 1/(1-s)^2 & \left(\text{MEMS}\right)\vspace{0.15cm}
\end{array}&&
\end{flalign*}

\begin{flushleft}
\textit{Classical solution}: $u\in C^2(\Omega)\cap C(\overline{\Omega})$ is a classical solution of \eqref{01} if it solves \eqref{01} in the classical sense (i.e. using the classical notion of derivative).
\end{flushleft}

\begin{flushleft}
\textit{Weak solution}: $u\in W_{0}^{1,2}(\Omega)$ is a weak solution of \eqref{01} if $0\leq u < s_0$ almost everywhere in $\Omega$ and $u=s_0$ in a subset with measure zero such that $f(u)\in L^{2}(\Omega)$ and
\begin{equation}\label{57}
\int_{\Omega}\left(\nabla_g u\cdot\nabla_g \phi + \phi A\cdot\nabla_g u\right)\, \mathrm{d} v_g = \lambda\int_{\Omega}f(u)\phi\, \mathrm{d} v_g,\quad \forall \phi\in W_{0}^{1,2}(\Omega).
\end{equation}
\end{flushleft}
We also consider \textit{weak subsolution} (\textit{weak supersolution}) in analogy with this definition. For instance, $u\in W_{0}^{1,2}(\Omega)$ is a weak subsolution of \eqref{01} if $0\leq u < s_0$ almost everywhere in $\Omega$ and $u=s_0$ in a subset with measure zero such that $f(u)\in L^{2}(\Omega)$ with $``\leq "$ $(``\geq ")$ instead of $``="$ in \eqref{57}.
\begin{flushleft}
\textit{Minimal solution}: For problem \eqref{01}, we say that a weak solution $u\in W_{0}^{1,2}(\Omega)$ is a minimal solution if $u \leq v$ almost everywhere for all $v$ supersolution. We denote minimal solution of \eqref{01} by $\underline{u}_\lambda$.
\end{flushleft}
\begin{flushleft}
\textit{Regular solution}: We say that a weak solution $u$ of \eqref{01} is a regular solution if $\sup_{\Omega} u <s_0.$
\end{flushleft}
\begin{flushleft}
\textit{Semi-stable solution}: We say that a classical solution $u$ of \eqref{01} is semi-stable solution provided that
\begin{equation}\label{bora}
\int_{\Omega} \left(|\nabla_g \xi|^{2} + \xi A(x)\cdot\nabla_g \xi\right) \mathrm{d} v_g \geq \int_{\Omega}\lambda f^\prime (u)\xi^2\mathrm{d} v_g,\quad\forall\xi\in C_{0}^{1}(\Omega).
\end{equation}
\end{flushleft}
Analogously one defines stable solution if we have the strict inequality in \eqref{bora}. We say that a classical solution $u$ of \eqref{01} is unstable if $u$ is not semi-stable.

We can now formulate our main results. Using some ideas in \cite{BREVAZ1997,MIGPUE1980}, we prove the existence of a critical parameter $\lambda^*$ which is related with the solvability of \eqref{01}. Moreover, we obtain upper and lower estimates for this critical parameter $\lambda^*$. This implies that the explosion threshold cannot drop arbitrarily close to zero, no matter what the field $\phi$ is. Different from \cite{BERKISNOVRYZ2010}, here we do not assume incompressibility of the flow, that is, $\nabla\cdot\phi=0$.

\begin{theorem}\label{22}
There exists a critical parameter
 $\lambda^*\in \mathbb{R}, \; \lambda^*>0$ such that 
\begin{description}
\item[$(i)$] For all $ \lambda \in (0, \lambda^*)$ problem \eqref{01} possesses  an unique minimal classical solution $\underline{u}_\lambda$ which is positive and semi-stable, and the map $\lambda\rightarrow \underline{u}_\lambda(x)$ is increasing on $(0,\lambda^*)$ for each $x\in\Omega.$ 
\item[$(ii)$] The following estimates hold
\[
\beta (1-\beta\max_{\overline{\Omega}}w)^2\leq\lambda^*\leq\lambda_1,
\]
where $w$ and $\beta$ are given in Lemma \ref{28} and $\lambda_1$ is the first eigenvalue of $-\Delta_g + A\cdot\nabla_g$ with zero Dirichlet boundary condition.
\item[$(iii)$] For $\lambda>\lambda^*$ there are no solutions, even in weak sense.
\item[$(iv)$] semi-stable solutions of \eqref{01} are necessarily minimal solutions.
\end{description}
\end{theorem}

In view of item $(i)$ of Theorem \ref{22}, we can define the function
\[
u^*(x):=\lim_{\lambda \nearrow \lambda^*}\underline{u}_\lambda(x),
\]
which is measurable, since it is a limit of measurable functions. If $u^*$ is a weak solution of \eqref{01} at  $\lambda=\lambda^*$ it will be called extremal solution.

An important question which has attracted a lot of attention is whether the extremal solution $u^*$ is a classical solution. Here we are going to prove regularity of the extremal solution $u^*$ if the dimension of $\mathcal{M}$ is below the critical dimension $N^*$. Here we stress the fact that the critical dimension depends only on the nonlinearity $f(s)$ and does not depend of the manifold $\mathcal{M}$, which is given precisely by, 
\begin{flalign}\label{64}
\begin{array}{lccclll}
(i)   &  N^* & = & 10 & \text{if} & f(s) = e^s       & \left(\text{Gelfand}\right)\vspace{0.15cm}\\ 
(ii)  &  N^* & = & 11 & \text{if} & f(s) = (1+s)^m   & \left(\text{Power-type}\right)\vspace{0.15cm}\\
(iii) &  N^* & = & 8  & \text{if} & f(s) = 1/(1-s)^2 & \left(\text{MEMS}\right)\vspace{0.15cm}
\end{array}&&
\end{flalign}

\begin{theorem}\label{13}
The extremal solution $u^*$ of $(P_{\lambda^*})$ is classical provided that $1\leq N < N^*$.
\end{theorem}

\begin{remark}
	To obtain regularity of the extremal solutions defined on domains of $\mathbb{R}^N$ usually make use of an argument based on Hardy-type inequality (see Subsection \ref{59}). It is well known that a complete open Riemannian manifold with non-negative Ricci curvature of dimension greater than or equal to three in which a Hardy inequality are satisfied are close to the Euclidean space. In view of this, to obtain regularity results of extremal solutions for any Riemannian manifold one have to use an argument free of Hardy-type inequality.
\end{remark}

We now study uniqueness at the critical parameter $\lambda^*$.

\begin{theorem}\label{36}
For dimension $1\leq N < N^*$, the extremal solution $u^*$ is the unique classical solution of $(P_{\lambda^*})$ among all weak solutions.
\end{theorem}

Next, we present a multiplicity result for this class of equations when $\lambda$ is smaller and close to the critical parameter $\lambda^*$. The proof is carried out by the Mountain Pass Theorem in the same spirit of \cite{espghoguo2009}.

\begin{theorem}\label{37}
Let $1\leq N< N^*$ and $A=\nabla_ga$. Then, there exists $\delta>0$ such that for any $\lambda\in(\lambda^*-\delta,\lambda^*)$ we have a second branch of solutions $U_\lambda$ given by mountain pass for $J_{\epsilon,\lambda}$ on $W_{0}^{1,2}(\Omega).$
\end{theorem}

\begin{remark}\label{02}
We can use elliptic estimates to see that any regular solution $u$ of \eqref{01} belongs to $C^{1,\alpha}(\overline{\Omega})$. For that, we cover $\mathcal{M}$ by coordinate neighbourhood and consider a partition of unity subordinate to this cover. By using Schauder estimates it is easy to prove that $u\in C^{2,\alpha}(\overline{\Omega})$ and consequently any regular solution of \eqref{01} is a classical solution (see \cite{GILTRU,HEB2000}).
\end{remark}

\begin{remark}
The class of semi-stable solutions includes local minimizers, minimal solutions, extremal solutions and certain class of solutions found between a sub and a supersolution.
\end{remark}

We also obtain qualitative properties for semi-stable solutions of problem \eqref{01} if $\Omega=\mathcal{B}_R$ is a geodesic ball of a Riemannian model and $A$ is a radial vector field. Precisely, we prove that such solutions are radially symmetric and decreasing.  We say that $u\in C^{2}(\mathcal{B}_R)$ is radially symmetric and decreasing if $u(x)=u(r),$ where $r=\mathrm{dist}(x,\mathcal{O})$ and $u_r(r)<0$ for all $r\in (0,R)$.

The class of Riemannian model $(\mathcal{M}, g)$ includes the classical space forms. Precisely, a manifold $\mathcal{M}$ of dimension $N\geq 2$ admitting a pole $\mathcal{O}$ and whose metric $g$ is given, in polar coordinates around $\mathcal{O}$, by
\begin{equation}\label{08}
\mathrm{d} s^2=\mathrm{d} r^2+\psi(r)^2\mathrm{d} \theta^2 \quad \text{for }r \in (0,R)\text{ and }\theta\in\mathbb{S}^{N-1},
\end{equation}
where $r$ is by construction the Riemannian distance between the point $P=(r,\theta)$ to the pole $\mathcal{O}$, $\psi$ is a smooth positive function in $(0,R)$ and $\mathrm{d}\theta^2$ is the canonical metric on the unit sphere $\mathbb{S}^{N-1}$. Note that our results apply to the important case of space forms, i.e., the unique complete and simply connected Riemannian manifold of constant sectional curvature $K_\psi$ corresponding to the choice of $\psi$ namely,
\begin{flalign*}
\begin{array}{lccclll}
(i)   &   \psi(r) & = & \sinh r,  & \quad  K_\psi=-1 & \quad \left(\text{Hyperbolic space}\right)\\ 
(ii)  &   \psi(r) & = & r,        & \quad  K_\psi=0  & \quad \left(\text{Euclidean space}\right)\\
(iii) &   \psi(r) & = & \sin r,   & \quad  K_\psi=1  & \quad \left(\text{Elliptic space}\right)
\end{array}&&
\end{flalign*}

Next, we state a radial symmetry result for semi-stable solutions defined on a geodesic ball of $\mathcal{M}$. The proof is based on the fact that any angular derivative of $u$ would be either a sign changing first eigenfunction of the linearized operator at $u$ or identically zero, thanks to the semistability condition. In addiction, the monotonicity of $u$ is due to the positivity of the nonlinearity.

\begin{theorem}\label{07}
If $u\in C^2(\mathcal{B}_R)$ is a classical stable solution of \eqref{01} with a radial vector field $A$, then $u$ is radially symmetric and decreasing.
\end{theorem}

We can improve the result of Theorem \ref{13} giving an estimate for the radial case.

\begin{theorem}\label{06}
Let $u$ be the extremal solution of $(P_{\lambda^*})$ on a geodesic ball $\mathcal{B}_r$ of a Riemannian model with $2\leq N < N^*$. Then $u^*$ is a classical solution and
\[
\Vert u^* \Vert_\infty\leq c,
\]
where $c>0$ is a constant which does not depends of $\lambda$. We emphasize that, for the case $f(u)=1/(1-u)^2$ we have $c<1.$
\end{theorem}

\subsection{Motivation and previous results}\label{59}
In order to motivate our results we begin by giving a brief survey on this subject. Singular elliptic problems of the form
\begin{equation}\label{56}
- L u =\lambda h(u),
\end{equation}
for a second order elliptic operator $L$ under various boundary conditions, has been extensively studied since the papers of D.~Joseph and T.~Lundgren \cite{JOSLUN1972}, J.~Keener and H.~Keller \cite{KEEKEL1974} and M. Crandall and P. Rabinowitz \cite{CRARAB1975, CRARAB1973}.  It has been shown in these pioneering works that there exists a critical threshold $\lambda^*>0$ such that \eqref{56} admits positive solutions for $0<\lambda<\lambda^*$, while no positive solutions exist for $\lambda>\lambda^*$. In \cite{MIGPUE1980}, F.~Mignot and J-P.~Puel studied regularity results to certain nonlinearities, namely,  $h(s)=e^s$, $g(s)=s^m$ with $m>1$, $g(s)=1/(1-s)^k$ with $k>0$. Very recently, this analysis was completed by N. Ghoussoub and Y. Guo \cite{GHOGUO2007} for the MEMS case, precisely, $-\Delta u=\lambda f(x)/(1-u)^2$ in a bounded domain $\Omega$ under zero Dirichlet boundary condition, among other results they proved that $N=8$ is the critical dimension for this class of problems.

H. Brezis and J. L. Vazquez in \cite{BREVAZ1997} treated the delicate issue of regularity of solutions at extreme value $\lambda=\lambda^*$ of 
\begin{equation}\label{5a}
- \Delta  u =\lambda h(u), \quad u=0, \; \partial \Omega,
\end{equation}
where the nonlinerity $h(s)$ is continuous, positive, increasing and convex function defined for $u\geq 0$ with $h(0)>0$ and 
$\lim_{s\rightarrow \infty}h(s)/s=+\infty$. Typical examples are $h(s)=e^s$ and $h(s)=(1+s)^p$, with $p>1$. The authors characterized the singular $H^1$ extremal solutions and the extremal value by a criterion consisting of two conditions: $(i)$ they must be energy solutions, not in $L^{\infty};$ $(ii)$ they must satisfy a Hardy inequality which translates the fact that the first eigenvalue of the linearized operator is nonnegative. In order to apply this characterization to those examples, they also establish a simultaneous generalization for Hardy's and Poincare's inequalities for all dimensions $N \geq 2$. 
 One of main theorem in \cite{BREVAZ1997} is focused on the application to the cases $h(s)=e^s$ and $h(s)=(1+s)^m$, $m>1$, in the unit ball of Euclidean space centred at the origin. If $h(s)=e^s$, the procedure shows that $u^*(x)=\textrm{log}|x|^{-2}$ is an unbounded extremal solution in $H_{0}^{1}$ for $\lambda^*=2(n-2)$ if and only if $n\geq 10$. A similar result is also obtained for the case $h(s)=(1+s)^m$, $m>1$.

For more general nonlinearities $h(s)$ and domains $\Omega$, regularity for solutions of \eqref{5a} at $\lambda=\lambda^*$ has been established by G. Nedev in \cite{NED2000}. Precisely, he proved that $u^*\in L^{\infty}(\Omega)$ if $N \leq 3$, while $u^*\in H^1_0(\Omega)$ if  $N \leq 5$, for every bounded domain $\Omega$ and convex nonlinearity $h(s)$ satisfying
\begin{equation}\label{qwe}
h\in C^1,\text{ nondecreasing, }h(0)>0,\text{ and }\lim_{s\rightarrow +\infty}\frac{h(s)}{s}=+\infty.
\end{equation}
After that X.~Cabr\'{e} in \cite{CAB2010} proved regularity if $N\leq 4$ assuming conditions \eqref{qwe} for convex and bounded domain $\Omega$. Under the same assumptions, X. Cabr\'{e} and M. Sanch\'{o}n in\cite{CABSAN2013} complete the analysis of regularity for dimensions $5 \leq N \leq 9$. In \cite{CAB2010,CABSAN2013} were not assumed convexity on the nonlinearity $h(s)$, but in contrast with Nedev's result, it was assumed convexity of the set $\Omega$.

The following equation has often been used to model a simple electrostatic Micro-Electro-Mechanical system (MEMS) device:
\begin{equation}\label{52}
\left\{
\begin{alignedat}{3}
-\Delta u = & \,\frac{\lambda f(x)}{(1-u)^2} & \quad \text{in} & \quad\Omega,        \\
u =         & \, 0                      &  \text{on}      & \quad\partial\Omega,\\
\end{alignedat}
\right.
\end{equation}
where $\Omega$ is a smooth bounded domain in $\mathbb{R}^N$, $f\in C^\alpha(\overline{\Omega})$, $\lambda>0$ is proportional to the applied voltage and $0<u(x)<1$ denotes the deflection of the membrane. We refer the reader to \cite{espghoguo2009,PELBER2002} for a recent survey on this subject. MEMS are often used to combine electronics with microsize mechanical devices in the design of various types of microscopic machinery. For MEMS equation \eqref{52}, non-existence results and upper bounds for the extremal parameter were established (see \cite{GUOPAN}) in terms of material and geometric properties of the membrane, results further complemented in \cite{GHOGUO2007}, where the existence of minimal solutions for $0<\lambda<\lambda^*$ was proved as well as the existence and uniqueness of the extremal solution for $\lambda=\lambda^*$ provided that dimension $1\leq N\leq 7$. In this dimensional range, the existence of non minimal solutions was obtained in \cite{espghoguo2007}, where the authors study the branch of semi-stable solutions and where existence results in higher dimensions, for a suitable hypothesis, were also established.

About the explosion problem in an incompressible flow we refer to the very interesting work of Berestycky et al. \cite{BERKISNOVRYZ2010} 
which considered the non-selfadjoint elliptic problem
\begin{equation}\label{27}
\left\{
\begin{alignedat}{3}
-\Delta u + \phi\cdot\nabla u= & \, \lambda g(u) & \quad \text{in} & \quad\Omega, \\
u = & \, 0  &  \text{on} & \quad \partial\Omega,\\
\end{alignedat}
\right.
\end{equation}
with a prescribed incompressible flow $\phi(x)$ so that $\nabla\cdot \phi = 0$ where $\Omega$ is a smooth bounded domain of Euclidean space. In this work the authors began to investigate how the presence of an underlying flow and its properties affect the explosion. They were interested in qualitative dependence of the critical explosion $\lambda^*$ with respect to the vector field $\phi$. As observed numerically in \cite{BERKAGJOUSIV1997} for a two dimensional cellular flow, the explosion threshold increases for flows oscillating on a small scale and it may actually decrease if the flow has large scale variations. This analysis motivated our result about upper and lower estimates for $\lambda^*$ for a general vector field $A$ (cf. Theorem \ref{22}).

Still on the non-selfadjoint elliptic problem $\eqref{27}$, C.~Cowan and N.~Ghoussoub \cite{COGH2008} proved regularity results for extremal solutions for the nonlinearities: $f(u)=1/(1-u)^2$ or $f(u)=e^u$ and a general class of advection term (not necessarily incompressible). The argument in \cite{COGH2008} was based on a class of Hardy type inequality contained in \cite{COW20102}. At this point we emphasize that a similar argument can not be applied for a general Riemannian manifold setting to prove regularity (cf. Theorem~\ref{13}), since it is known that the existence of Hardy or Gagliardo-Nirenberg or Caffarelli-Kohn-Nirenberg inequality on a Riemannian manifold implies qualitative properties on the Riemannian manifold. Precisely, it was shown that if $\left(\mathcal{M},g\right)$ is a complete Riemannian manifold with nonnegative Ricci curvature in which a Hardy or Gagliardo-Nirenberg or Caffarelli-Kohn-Nirenberg type inequalities holds then $\mathcal{M}$ is close to Euclidean space in some suitable sense, see \cite{MAN}.

X.~Luo, D.~Ye and F.~Zhou in \cite{LUYEZH2011} studied \eqref{27} where $h:[0,a)\rightarrow\mathbb{R_+}$ with fixed $a\in(0,+\infty)$ satisfies the following condition:
\begin{equation}\label{42}\tag{H}
h\text{ is }C^2\text{, positive, nondecreasing and convex in }[0,a)\text{ with }\lim_{s\rightarrow a^-}h(s)=+\infty.
\end{equation}

The authors in \cite{LUYEZH2011} observed a close similarity between \eqref{27} and the Emden-Fowler equation with superlinear regular nonlinearity,
\begin{equation}\label{38}
\left\{
\begin{alignedat}{3}
-\Delta u = & \, \lambda h(u) & \quad \text{in} & \quad\Omega, \\
u = & \, 0  &  \text{on} & \quad \partial\Omega,\\
\end{alignedat}
\right.
\end{equation}
with $\lambda>0$ and $h:[0,+\infty)\rightarrow [0,+\infty)$ such that 
\begin{equation}\label{44}
h\text{ is }C^2,\text{ nondecreasing convex and }\lim_{s\rightarrow+\infty}\frac{h(s)}{s}=+\infty.
\end{equation}
The problem \eqref{38} can be linked to \eqref{27} where $g:[0,a)\rightarrow\mathbb{R_+}$ with fixed $a\in(0,+\infty)$ satisfies \eqref{42}. Without loss of generality, we can fix $a=1$ and consider the transformation $v=-\mathrm{ln}(1-u)$. Thus let $u$ solving \eqref{27} then $v$ verifies
\begin{equation}\label{43}
\left\{
\begin{alignedat}{3}
-\Delta v+|\nabla v|^2+c(x)\cdot\nabla v = & \, \lambda e^vg(1-e^{-v}):=\lambda h(v) & \quad \text{in} & \quad\Omega, \\
v = 0 \hspace{3.9cm}&  &  \text{on} & \quad \partial\Omega,\\
\end{alignedat}
\right.
\end{equation}
Therefore $h$ satisfies \eqref{44} and $v^*=-\mathrm{ln}(1-u^*)$ is the extremal solution for the problem \eqref{43}. Thus the regularity of $u^*$ is equivalent to the boundedness of $v^*$. We mention that the situation could be very different with the presence of advection terms, see \cite{WEIYE2010,COGH2008}. If the vector field $\phi$ nontrivial the operator $-\Delta+\phi\cdot\nabla$ is not self-adjoint. However if $\phi=-\nabla\gamma$ in $\Omega$ then $-\Delta+\phi\cdot\nabla$ can be rewritten as a self-adjoint operator of the form $-e^{-\gamma}\mathrm{div}(e^{\gamma}\nabla)$. In that case, \eqref{27} admits a variational structure and we can expect more precise estimates of minimal solutions $u_\lambda$, as in the radial case.

Our Theorems \ref{22} and \ref{36} improve and complement some results in \cite{LUYEZH2011} for the nonlinearities described in \eqref{60}. In Theorem \ref{13} we determine the critical dimension $N^*$ and prove regularity for extremal solutions if $N<N^*$. This theorem is close related with Theorem 1.3, 1.4 and 1.5 in \cite{LUYEZH2011} where X.~Luo et al investigated the regularity of extremal solutions of \eqref{27} and it was proved that if $f$ satisfies \eqref{42} and some additional assumptions, then $u^*$ is regular if $N\leq N_0$, where $N_0$ depends on $f$ and in the most significant cases is less than $10$. In our Theorem \ref{06} we do not required that the advection term $\phi$ is the gradient of a smooth radial function, even for $N=2$, to prove our regularity result for radial case for the nonlinearities described in \eqref{60}, differently of Theorem 1.1 in \cite{LUYEZH2011}.

In the past decades, there have been considerable attentions to be paid on the research of singular elliptic problems defined on Riemannian manifolds. A.~Farina, L.~Mari and E.~Valdinoci \cite{FARMARVAL} studied Riemannian manifolds with non-negative Ricci curvature that posses a stable, nontrivial solution of a semilinear equation
\begin{equation}\label{61}
-\Delta_gu=f(u).
\end{equation}
Under suitable assumptions, it was proved symmetry results for the solutions and the rigidity of the underlining manifold. E.~Berchio, A.~Ferrero and G.~Grillo in \cite{BERFERGRI2014} studied existence, uniqueness and stability of radial solutions of \eqref{61} for the Lane-Emden-Fowler equation, that is, $f(s)=|s|^{p-1}s$, on a Riemannian manifold of dimension $N\geq 3$ with a pole. In addiction, the authors obtained that the sign properties and asymptotic behavior of solutions are influenced by the critical Sobolev exponent while the so-called Joseph-Lundgren exponent is involved in the stability of solutions. D.~Castorina and M.~Sanch\'{o}n \cite{CASSAN2013} condidered the problem \eqref{61} in a geodesic ball $\mathcal{B}_R$ with zero Dirichlet boundary condition of a Riemannian model $\mathcal{M}$. They proved radial symmetry and monotonicity for the class of semi-stable solutions. Moreover, they establish $L^\infty$, $L^p$ and $W^{1,p}$ estimates which are optimal and do not depend on the nonlinearity $f(s)$. As an application, under standard assumptions on the nonlinearity $f(s)$, they proved that the extremal solution $u^*$ is bounded whenever $N \leq 9$ and they studied the extremal solution for some exponential and power nonlinearities using an improved Hardy inequality to establish the optimality of their regularity results.

\subsection{Outline} The paper is organized as follows. In the next section we bring a version of Maximum principle, we prove a Sub- and Super-solution method and a version of Hodge-Helmholtz decomposition. In Section \ref{49}, we study the existence of extremal parameter $\lambda^*$ and minimal solutions $\underline{u}_\lambda$. The Section \ref{50} is devoted to prove monotonicity results for the branch of minimal solutions and $L^p$-estimates for $\underline{u}_\lambda$ uniformly in $\lambda$, and we determine the critical dimensions for this class of problems for singular nonlinearities of type MEMS, Gelfand and power case. By using this estimates, we prove that the extremal solution $u^*:=\lim_{\lambda \nearrow \lambda^*}\underline{u}_\lambda$ is classical whenever the dimension of $\mathcal{M}$ is below the critical dimension. In Section \ref{63}, for the particular case where $\mathcal{M}$ is a Riemannian model and $\Omega$ is a geodesic ball of $\mathcal{M}$, we establish symmetry and monotonicity for the class of semi-stable solutions and we also prove $L^\infty$-estimates for $u^*$. In Section \ref{66} we analyze the branch of minimal solutions and we prove multiplicity of solutions if $\lambda\in(\lambda^*-\delta,\lambda^*)$ for some $\delta>0$ and uniqueness at $\lambda^*.$

\section{Key-ingredients} We use a Comparison Principle for weak solutions of quasilinear elliptic differential equation in divergence form on complete Riemannian manifold. We need a simple version of Theorem 3.3 found in \cite{ANTMUGPUC2007}.

\begin{proposition}[Maximum Principle]
Let $w$ a weak supersolution of $-\Delta_g u +A\cdot\nabla u=0$. If $w\geq 0$ on $\partial\Omega$, then $w\geq 0$ in $\Omega.$
\end{proposition}

For the sake of completeness, we prove the Sub and Supersolution result in Proposition \ref{12} using the Monotone Iteration Method. In this way, T. Kura \cite{KUR1989} has proved many results about the existence of a solution between sub and supersolutions for quasilinear problems.

\begin{proposition}[The sub- and super-solution method]\label{12}
Let $\underline{u}$ and $\overline{u}$ subsolution and supersolution of \eqref{01}, respectively, that satisfies $\underline{u}\leq\overline{u}$ a.e. in $\Omega$. Then problem \eqref{01} has a weak solution $u$ such that $\underline{u}\leq u \leq \overline{u}$ a.e. in $\Omega$.
\end{proposition}

\begin{proof}
Denote by $u_0=\underline{u}$. We define a sequence $u_n$ inductively where each $u_n$ is the unique weak solution of the problem
\begin{equation}\label{09}
\left\{
\begin{alignedat}{3}
-\Delta_g u_n + A\cdot\nabla_g u_n+cu_n= & \,\lambda f(u_{n-1})+cu_{n-1} & \quad \text{in} & \quad\Omega, \\
u_n = & \, 0  &  \text{on} & \quad \partial\Omega.\\
\end{alignedat}
\right.
\end{equation}
This sequence satisfies $\underline{u}\leq u_{n-1}\leq u_n \leq \overline{u}$. In fact, consider \eqref{09} where $n=1.$ We have $u_1\in W_{0}^{1,2}(\Omega)$ and by Maximum Principle follows $\underline{u}\leq u_1 \leq \overline{u}.$ In the same way $u_2\in W_{0}^{1,2}(\Omega)$ and satisfies $\underline{u}\leq u_1\leq u_2\leq \overline{u}$. By induction we have the result i.e., $\underline{u}\leq u_{n-1}\leq u_n \leq \overline{u}$. Now, observe that $u_n$ is bounded in $W_{0}^{1,2}(\Omega)$ and has a subsequence that converges weakly to $u\in W_{0}^{1,2}(\Omega)$. Taking the limit in the equation follows that $u$ is a weak solution of the problem
\[
\left\{
\begin{alignedat}{3}
-\Delta_g u + A\cdot\nabla_g u= & \,\lambda f(u) & \quad \text{in} & \quad\Omega, \\
u = & \, 0  &  \text{on} & \quad \partial\Omega.\\
\end{alignedat}
\right.
\]
\end{proof}

We also have a version of Hodge-Helmholtz decomposition in order to deal with general vector fields $A.$ This decomposition of vector fields is one of the fundamental theorem in fluid dynamics. It describes a vector field in terms of its divergence-free and rotation-free components. For more results in this subject we refer the reader to \cite{PRE2008}.

\begin{lemma}\label{48}
Any vector field $A\in C^{\infty}(\overline{\Omega},TM)$ can be decomposed as $A=-\nabla_ga + C$ where $a$ is a smooth scalar function and $C$ is a smooth bounded vector field such that $\mathrm{div} (e^{a}C)=0$. 
\end{lemma}

\begin{proof}
Let $\nu$ the unit outer normal on $\partial\Omega.$ Using Krein-Rutman theorem, we can find a positive solution $w$ of
\[
\left\{
\begin{alignedat}{3}
\Delta_g w + \mathrm{div}(wA)= & \,  \mu w & \quad \text{in} & \quad\Omega, \\
(\nabla_gw + wA)\cdot\nu= & \, 0  &  \text{on} & \quad \partial\Omega.\\
\end{alignedat}
\right.
\]
for a constant $\mu\in \mathbb{R}$. Integrating the equation over $\Omega$ one sees that $\mu=0.$ By the maximum principle, $w$ is positive up to the boundary. Now define $a:=\log(w)$ and $C:=A+\nabla_g a$. It is easy to see that $\mathrm{div} (e^aC)=e^a\nabla_g a \cdot A +e^a\mathrm{div} C =\nabla_g w\cdot A +|\nabla_g w|^2/w+w\mathrm{div} A -|\nabla_g w|^2/w+\Delta w= 0 .$ 
\end{proof}

\section{Existence results}\label{49}

Now we can construct a supersolution for the problem \eqref{01} if $\lambda$ is sufficient small.

\begin{lemma}\label{28}
Let $w\in W_{0}^{1,2}(\Omega)$ be a weak solution of the problem
\begin{equation}\label{10}
\left\{
\begin{alignedat}{3}
-\Delta_g w + A\cdot\nabla_g w= & \, 1 & \quad \text{in} & \quad\Omega, \\
w = & \, 0  &  \text{on} & \quad \partial\Omega.\\
\end{alignedat}
\right.
\end{equation}
There exist $\beta > 0$ such that $\beta w$ is a supersolution of \eqref{01} for $\lambda$ sufficient small.
\end{lemma}

\begin{proof}
For a large $c>0$, let $\tilde{\mathcal{L}} w=-\Delta_g w + A\cdot\nabla_g w +cw$ and consider the problem
\[
\left\{
\begin{alignedat}{3}
\tilde{\mathcal{L}} w= & \, f & \quad \text{in} & \quad\Omega, \\
w = & \, 0  &  \text{on} & \quad \partial\Omega.\\
\end{alignedat}
\right.
\]
If we write $w=\tilde{\mathcal{L}}^{-1}(1+cw)=\mathcal{N}(w)$ we can use Schauder Fixed Point Theorem to find a solution $w$ of \eqref{10}. By elliptic estimates $w\in C^{1}(\overline{\Omega})$ so we can take $\beta > 0$ such that $\beta\max_{\overline{\Omega}}w < s_0.$ If $\lambda\leq \beta/f(\beta\max_{\overline{\Omega}}w)$ we have
\[
\displaystyle\int_{\Omega}\left(\nabla_g(\beta w)\nabla_g\phi + \phi A\cdot\nabla_g (\beta w)\right) \,\mathrm{d} v_g = \beta\int_{\Omega}\phi\geq\int_{\Omega}\lambda\phi f(\beta w)\,\mathrm{d} v_g,
\]
i.e., $\beta w$ is a supersolution of \eqref{01}.
\end{proof}

Let us define
$$\Lambda:=\lbrace\lambda\geq 0\text{ : }\eqref{01}\text{ has a classical solution}\rbrace.$$

We can define the extremal parameter $$\lambda^*=\sup\Lambda.$$

\begin{remark}
Using Lemma \ref{28} we can find a regular solution between $0$ and $\beta w$. With this, $\sup \Lambda>0.$
\end{remark}

\begin{lemma}
The set $\Lambda$ is a interval.
\end{lemma}

\begin{proof}
Initially, we prove that $\Lambda$ does not consist of just $\lambda=0$. Let $u$ a classical solution for problem \eqref{01} with $\lambda<\lambda^*.$ Observe that $u_0=0$ and $u$ are sub and supersolution, respectively, for the problem \eqref{01}. Using the Sub and Supersolution Method, there exist a weak solution $v\in W_{0}^{1,2}(\Omega)$ such that $u_0 \leq v<u < s_0$. By Remark \ref{02}, $v$ is a classical solution. This solution is a supersolution for ($\overline{P}_\mu$) if $\mu\in (0,\lambda)$. Again, there exist a classical solution for the problem ($\overline{P}_\mu$). Thus, $\Lambda$ is a interval.
\end{proof}

\begin{lemma}\label{23}
The interval $\Lambda$ is bounded.
\end{lemma}

\begin{proof}
Suppose that exist a classical solution $u$ of $\eqref{01}$, for $\lambda$ sufficiently large. We can suppose that $\lambda > \lambda_1$, where $\lambda_1$ is the first eigenvalue associate to the operator $L=-\Delta_g + A\cdot\nabla_g$. Let $v_1$ the first eigenvalue in $\lambda_1$, i.e.,
\[
\left\{
\begin{alignedat}{3}
-\Delta_g v_1 + A\cdot\nabla_g v_1= & \, \lambda_1 v_1 & \quad \text{in} & \quad\Omega, \\
v_1 = & \, 0  &  \text{on} & \quad \partial\Omega.\\
\end{alignedat}
\right.
\]
By regularity theory, follows that $v_1\in C^{1,\alpha}(\overline{\Omega})$. By homogeneity, we can suppose $\Vert v_1 \Vert_\infty < 1.$ So $v_1$ and $u$ satisfies
\[
-\Delta_g v_1 + A\cdot\nabla_g v_1 = \lambda_1 v_1 < \lambda f(u)=-\Delta_g u + A\cdot\nabla_g u.
\]
By Comparison Principle follows that $v_1\leq u.$ Now, given $\epsilon >0$, we take $v_2$ a solution of
\[
\left\{
\begin{alignedat}{3}
-\Delta_g v_2 + A\cdot\nabla_g v_2 = & \, (\lambda_1+\epsilon) v_1 & \quad \text{in} & \quad\Omega, \\
v_2 = & \, 0  &  \text{on} & \quad \partial\Omega.\\
\end{alignedat}
\right.
\]
As above, $v_1\leq v_2\leq u$. By induction, we have solutions $v_n$ such that
\[
\left\{
\begin{alignedat}{3}
-\Delta_g v_n + A\cdot\nabla_g v_n= & \, (\lambda_1+\epsilon) v_{n-1} & \quad \text{in} & \quad\Omega, \\
v_n = & \, 0  &  \text{on} & \quad \partial\Omega,\\
\end{alignedat}
\right.
\]
with $v_1 \leq ... \leq v_{n-1} \leq v_n \leq u$ in $C^{1,\alpha}(\overline{\Omega})$. So, $v_n \rightharpoonup v$ in $W_{0}^{1,2}(\Omega)$. It follows that $v$ satisfies
\[
\left\{
\begin{alignedat}{3}
-\Delta_g v + A\cdot\nabla_g v= & \, (\lambda_1+\epsilon) v & \quad \text{in} & \quad\Omega, \\
v = & \, 0  &  \text{on} & \quad \partial\Omega.\\
\end{alignedat}
\right.
\]
This is impossible since the first eigenvalue is isolated.
\end{proof}

\begin{remark}
Clearly, $\lambda^*<+\infty$ and there are no classical solution of \eqref{01} for $\lambda > \lambda^*$.
\end{remark}

\section{Monotonicity results and estimates for minimal solutions}\label{50}

\subsection{Minimal solutions}

\begin{lemma}\label{25}
For each $\lambda < \lambda^*$ there exist a unique minimal solution $\underline{u}_\lambda$ for the problem \eqref{01}. Therefore, for all $x\in\Omega,$ the map $\lambda \rightarrow \underline{u}_\lambda (x)$ is strictly increasing.
\end{lemma}

\begin{proof}
Consider the weak solution $u$ given by Proposition \ref{12}. All supersolutions $v$ of \eqref{01} satisfies $u\leq v.$ Thus $u$ is minimal. The uniqueness follows by minimality of $u$. In this way, we define $u:=\underline{u}_\lambda.$ Therefore, if $\lambda < \mu,$ we have that $\underline{u}_\mu$ is a supersolution of \eqref{01}. Thus, $\underline{u}_\lambda < \underline{u}_\mu.$
\end{proof}

Let $u$ be a semi-stable solution of \eqref{01}, and let us consider the following  eigenvalue problem involving the linearized operator $L_{u,\lambda}=-\Delta_g+A\cdot\nabla_g -\lambda f^\prime (u)$ at $u$, 
\[
\left\{
\begin{alignedat}{3}
 L_{u,\lambda} \phi  = &\,   \mu \phi  & \text{ in } & \Omega\\
 u  = & \, 0 & \quad \text{ on }  & \partial\Omega.
\end{alignedat}
\right.
\]
It is well known that there exists a smallest positive eigenvalue $\mu$, which we denote by $\mu_{1,\lambda}$, and an associated eigenfunction $\phi_{1,\lambda}>0$ in $\Omega$, and $\mu_{1,\lambda}$ is a simple eigenvalue and has the following variational characterization 
\[
\mu_{1,\lambda}=\inf \left\{ \langle L_{u,\lambda}\phi,\phi \rangle_{L^2(\Omega)} : \phi\in W_{0}^{1,2}(\Omega), \int_{\Omega}\phi^2 \, \mathrm{d} v_g=1 \right\}.
\]

\begin{lemma}\label{24}
If $0\leq \lambda<\lambda^*$, the minimal solutions are semi-stable.
\end{lemma}

\begin{proof}
Let $\underline{u}_\lambda$ minimal solution of \eqref{01}. Suppose that $\underline{u}_\lambda$ is not semi-stable i.e., the first eigenvalue $\mu_{1,\lambda}$ of operator $L_{u,\lambda}$ is negative. Consider the function $\psi_\epsilon =\underline{u}_\lambda - \epsilon \psi \in W_{0}^{1,2}(\Omega)$, where $\psi\in W_{0}^{1,2}(\Omega)$ is the first positive eigenvalue of operator $-\Delta_g + A\cdot\nabla_g-\lambda f^\prime (\underline{u}_\lambda)$. Using Taylor's formula, for $\Vert\xi\Vert_{W^{1,2}_{0}(\Omega)}$ sufficiently small we have
\[
\begin{alignedat}{2}
-\Delta_g\psi_\epsilon + A\cdot\nabla_g\psi_\epsilon-\lambda f(\psi_\epsilon) &= -\epsilon\lambda f(\psi)+\lambda f(\underline{u}_\lambda) - \epsilon\kappa\psi-\epsilon\lambda f^\prime(\underline{u}_\lambda)\psi\\
 &= -\epsilon\kappa\psi-\lambda\epsilon^2 f^{\prime\prime}(\xi)\psi^2\geq 0,
\end{alignedat}
\]
for $\epsilon$ sufficiently small, because $\kappa < 0.$ Thus $\psi_\epsilon$ is a supersolution of \eqref{01} and, by minimality of $\underline{u}_\lambda$ we have a contradiction.
\end{proof}

\begin{proof}[Proof of Theorem \ref{22}]
(1) The existence of $\lambda^*$ follows from Lemma \ref{23}. By Lemmas \ref{25} and \ref{24}, there exists $\underline{u}_\lambda$ minimal solution of \eqref{01} which is semi-stable and the function $\lambda \rightarrow \underline{u}_\lambda (x)$ is strictly increasing.\\
(2) Note that since $u_0=0$ is a subsolution of \eqref{01}, $\underline{u}_\lambda$ is non negative. In the same way, since a classical solution of \eqref{01} is also a supersolution, it follows that $\underline{u}_\lambda$ is a classical solution. The estimate is a consequence of Lemma \ref{28} and Lemma \ref{23}. \\
(3) Let $u_\mu$ a weak solution of $(P_\mu)$ with $\lambda^*<\mu$. Observe that $w=(1-\epsilon)u_\mu$ is a weak solution of $-\Delta_g w+A\cdot\nabla_g w=(1-\epsilon)\mu f(u_\mu)$, that is,
\[
\int_{\Omega}\left(\nabla_g w\cdot\nabla_g \phi+\phi A\cdot\nabla_g w\right)\,\mathrm{d} v_g=(1-\epsilon)\mu\int_{\Omega}\phi f(u_\mu)\,\mathrm{d} v_g. 
\]
An easy calculation shows that $w$ is a supersolution for $(P_{(1-\epsilon)\mu})$. Thus there exist a weak solution $v\leq w$. Since $v\leq w < u_\mu$, it follows that $v$ is a classical solution of $(P_{(1-\epsilon)\mu})$. If $\epsilon$ is sufficiently small, $\lambda^*<(1-\epsilon)\mu$. But this is a contradiction. Furthermore, since $u^*$ is a monotone limit of measurable functions, it is also measurable. \\
(4) Now, to prove that a semi-stable solution of \eqref{01} is minimal, let $u$ and $v$ a semi-stable solution and a supersolution of \eqref{01} respectively. For $\theta\in [0,1]$ and $0\leq\phi\in W_{0}^{1,2}(\Omega)$, we have
\begin{eqnarray*}
I_{\theta,\phi}:= \int_{\Omega}\left(\nabla_g(\theta u +(1-\theta)v)\cdot\nabla_g\phi + \phi A\cdot\nabla_g (\theta u +(1-\theta)v)\right)\,\mathrm{d} v_g\\
-\lambda\int_{\Omega}\phi f(\theta u +(1-\theta)v)\,\mathrm{d} v_g\geq 0,
\end{eqnarray*}
due to the convexity of function $s\rightarrow f(s).$ Since $I_{1,\phi}=0,$ the derivative of $I_{\theta,\phi}$ at $\theta=1$ is non positive, that is
\[
\int_{\Omega}\left(\nabla_g(u-v)\cdot\nabla_g\phi+\phi A\cdot\nabla_g(u-v)\right)\,\mathrm{d} v_g-\int_{\Omega}\lambda(u-v)\phi f^{\prime}(u)\,\mathrm{d} v_g\leq 0,\quad \forall\phi\geq 0.
\]
Testing $\phi=(u-v)^+$ and using that $u$ is semi-stable we get that
\[
\int_{\Omega}\left(|\nabla_g(u-v)^+|^2+(u-v)^+A\cdot\nabla_g(u-v)\right)\,\mathrm{d} v_g-\int_{\Omega}\lambda(u-v)(u-v)^+ f^\prime (u)\,\mathrm{d} v_g= 0,
\]
for all $\phi\geq 0$. Since $I_{\theta,(u-v)^+}\geq 0$ for any $\theta\in [0,1]$ and $I_{1,(u-v)^+}=\partial_\theta I_{1,(u-v)^+}=0,$ we have
\[
\partial_{\theta\theta}^{2}I_{1,(u-v)^+}=-\int_{\Omega}\lambda(u-v)^2(u-v)^+f^{\prime\prime}(u)\,\mathrm{d} v_g\geq 0.
\]
Clearly we have $(u-v)^+=0$ a.e. in $\Omega$ and therefore $\int_{\Omega}|\nabla_g(u-v)^+|^2\,\mathrm{d} v_g=0$, from which we conclude that $u\leq v$ a.e. in $\Omega$.
\end{proof}

\subsection{Determining the critical dimension}
For the MEMS case, the next lemma is the principal estimate, which was already behind the proof of the regularity of semi-stable solutions in dimensions lower than 7.

When $A=-\nabla_g a+C$ the problem $\eqref{01}$ can be rewritten as
\[
-\mathrm{div}_g(e^a\nabla_g u)+e^aC\cdot\nabla_gu=\frac{\lambda e^a}{(1-u)^2}.
\]
Thus the semi-stability and weak solution conditions becomes, respectively
\begin{equation}\label{46}
\int_{\Omega}\left( e^a|\nabla_g \eta|^2\ + e^a\eta C\cdot\nabla_g\eta\right) \, \mathrm{d} v_g  \geq \int_{\Omega}\frac{2\lambda e^a}{(1-u)^3}\eta^2\,\mathrm{d} v_g
\end{equation}
and
\begin{equation}\label{47}
\int_{\Omega}\left( e^a\nabla_g u\cdot\nabla_g \phi+ e^a\phi C\cdot\nabla_gu\right) \,\mathrm{d} v_g=\int_{\Omega}\frac{\lambda e^a\phi}{(1-u)^2} \,\mathrm{d} v_g.
\end{equation}
\begin{lemma}\label{03}
If $u$ is a semi-stable solution of \eqref{01} with $0<\lambda<\lambda^*$, $f(u)=1/(1-u)^2$ and $0<t<2+\sqrt{6}$, holds the following estimate
\[
\Vert e^{2a/(2t+3)}(1-u)^{-2} \Vert_{L^{t+3/2}}\leq \left[ \frac{4(2t+1)}{2+4t-t^2} \right]^{2/t}C_1\Vert\Omega\Vert^{2/(2t+3)}.
\]
\end{lemma}

\begin{proof}
Let $0<t<2+\sqrt{6}$ and $u$ semi-stable solution of \eqref{01}. Taking $\eta:=(1-u)^{-t}-1$ and using the Hodge-Helmholtz decomposition (see Lemma \ref{48}) we have $\int_{\Omega}e^a\eta C\cdot\nabla_g\eta \, \mathrm{d} v_g=0$. Thus, testing $\eta$ in the semistability condition \eqref{46}, we obtain
\[
\begin{alignedat}{2}
0 &\leq \displaystyle\int_{\Omega}e^a\lbrace t^2(1-u)^{-2t-2}|\nabla_g u|^{2}-2\lambda(1-u)^{-3}\left[ (1-u)^{-t}-1 \right]^2\rbrace \mathrm{d} v_g\\
&=\int_{\Omega}e^a\lbrace t^2(1-u)^{-2t-2}| \nabla_g u|^{2}-2\lambda(1-u)^{-2t-3}+4\lambda(1-u)^{-t-3}-2\lambda(1-u)^{-3}\rbrace \mathrm{d} v_g,\\
\end{alignedat}
\]
which implies
\begin{equation}\label{15}
-\int_{\Omega}e^a(1-u)^{-2t-2}|\nabla_g u|^{2}\,\mathrm{d} v_g \leq -\frac{2\lambda}{t^2}\int_{\Omega}e^a(1-u)^{-2t-3}\,\mathrm{d} v_g+\frac{4\lambda}{t^2}\int_{\Omega}e^a(1-u)^{-t-3}\,\mathrm{d} v_g.
\end{equation}
Testing $\phi:=(1-u)^{-2t-1}-1$ in the weak solution condition \eqref{47}, we obtain
\begin{equation}\label{16}
\begin{alignedat}{2}
\int_{\Omega}e^a|\nabla_g u|^{2}(2t+1)(1-u)^{-2t-2}\,\mathrm{d} v_g &= \int_{\Omega}\left\{\lambda e^a(1-u)^{-2t-3}-\lambda e^a(1-u)^{-2}\right\}\,\mathrm{d} v_g \\
&\leq\int_{\Omega}\lambda e^a(1-u)^{-2t-3}\,\mathrm{d} v_g,
\end{alignedat}
\end{equation}
because with this choice of $\phi$ we can check that $\int_{\Omega}e^a\phi\, C\cdot\nabla_gu \,\mathrm{d} v_g=0$. Using \eqref{16} and \eqref{15} we have
\[
-\frac{1}{2t+1}\int_{\Omega}e^a(1-u)^{-2t-3}\,\mathrm{d} v_g\leq -\frac{2}{t^2}\int_{\Omega}e^a(1-u)^{-2t-3}\,\mathrm{d} v_g+\frac{4}{t^2}\int_{\Omega}e^a(1-u)^{-t-3}\,\mathrm{d} v_g
\]
and follows that
\[
\left( \frac{2}{t^2}-\frac{1}{2t+1} \right)\int_{\Omega}e^a(1-u)^{-2t-3}\,\mathrm{d} v_g \leq \frac{4}{t^2}\int_{\Omega}e^a(1-u)^{-t-3}\,\mathrm{d} v_g.
\]
Using H\"{o}lder inequality with conjugate exponents $(2t+3)/(t+3)$ and $(2t+3)/t$, we have
\[
\begin{alignedat}{2}
\left( \frac{2}{t^2}-\frac{1}{2t+1} \right)\int_{\Omega}&e^a(1-u)^{-2t-3}\,\mathrm{d} v_g
\\ & \leq \frac{4}{t^2}\Vert\Omega\Vert^{t/(2t+3)}\left[\int_{\Omega}e^{a(2t+3)/(t+3)}(1-u)^{-2t-3}\,\mathrm{d} v_g\right]^{(t+3)/(2t+3)}\\
& \leq \frac{4C_1}{t^2}\Vert\Omega\Vert^{t/(2t+3)} \left[\int_{\Omega}e^a(1-u)^{-2t-3}\,\mathrm{d} v_g\right]^{(t+3)/(2t+3)},
\end{alignedat}
\]
where $C_1=\displaystyle\left[\sup_{\overline{\Omega}}e^{at/(t+3)}\right]^{(t+3)/(2t+3)}$. Thus,
\[
\left(\frac{2}{t^2}-\frac{1}{2t+1}\right)\left[\int_{\Omega}e^a(1-u)^{-2t-3}\,\mathrm{d} v_g\right]^{t/(2t+3)}\leq \frac{4C_1}{t^2}\Vert\Omega\Vert^{t/(2t+3)}
\]
and therefore
\[
\parallel e^{2a/(2t+3)}(1-u)^{-2}\parallel_{L^{t+3/2}}\leq \left[\frac{4(2t+1)}{2+4t-t^2}\right]^{2/t}C_1\Vert\Omega\Vert^{2/(2t+3)},
\]
this is the desired estimate.
\end{proof}

\begin{remark}
In the above estimate we used that $2(2t+1)>t^2$ which is an immediately consequence of our assumption $0<t<2+\sqrt{6}$.
\end{remark}

\begin{remark}
By the above estimate, $e^a(1-\underline{u}_\lambda)^{-2}$ is bounded uniformly in $\lambda$ over $L^p(\Omega)$ for all $p<p_0:=7/2+\sqrt{6}.$ By elliptic estimates, $\underline{u}_\lambda$ is uniformly bounded in $W_{0}^{1,p}(\Omega)$. Thus $u^*$ is a weak solution of $(P_{\lambda^*})$. If we take the limit $\lambda\nearrow\lambda^*$, we obtain the same $L^{p}$ estimate to extremal solution $u^*$.
\end{remark}
\begin{proposition}\label{51}
If $1\leq N \leq 7$ then $u^*$ is a classical solution of $(P_{\lambda^*})$.
\end{proposition}

\begin{proof}
Note that $e^a(1-u^*)^{-2}\in L^{3N/4}(\Omega)$. By elliptic regularity we have $u^*\in W_{0}^{2,3N/4}(\Omega)$ and by Sobolev immersion $u^*\in C^{0,2/3}(\overline{\Omega}).$ If we suppose that $\Vert u^*\Vert_\infty =1$, there exist a element $x_0\in\Omega$ such that $u^*(x_0)=1$. Since $| 1-u^*(x)|\leq C \mathrm{dist}(x,x_0)^{2/3}$ we have,
\[
\frac{e^{a/2}}{1-u^*(x)}\geq\frac{Ce^{a/2}}{\mathrm{dist}(x,x_0)^{2/3}},
\]
and hence
\[
\infty > \int_{\Omega}\frac{e^{3Na/4}}{((1-u^*)^2)^{3N/4}}\,\mathrm{d} v_g\geq C\displaystyle\inf_{x\in\overline{\Omega}}\lbrace e^{a/2}\rbrace \int_{\Omega}\frac{1}{\mathrm{dist}(x,x_0)^N}\,\mathrm{d} v_g=\infty.
\]
This is a contradiction. Thus $\Vert u^* \Vert_{\infty}<1$. This implies $e^a(1-u^*)^{-2}\in L^\infty(\Omega)$ and $u^*$ is a classical solution of $(P_{\lambda^*})$.
\end{proof}

With a slight variation of the above arguments, the same approach works on the Gelfand and Power-type cases.

\begin{lemma}\label{39}
If $u$ is a semi-stable solution of \eqref{01} with $0<\lambda<\lambda^*$, $f(u)=e^u$ and $0<t<2$, holds the following estimate
\[
\parallel e^{1/(2t+1)a+u}\parallel_{L^{2t+1}}\leq \left[\frac{2t}{2t-t^2}\right]^{1/t}C_1\Vert\Omega\Vert^{1/(2t+1)}.
\]
\end{lemma}

\begin{proof}
Let $0<t<2$ and $u$ semi-stable solution of \eqref{01}. Define $\eta:=e^{tu}-1$ and $\phi:=e^{2tu}-1$. Testing $\eta$ in the semistability condition we have
\[
0 \leq \int_{\Omega}t^2e^{a+2tu}|\nabla_g u|^{2}-\lambda e^{a+(2t+1)u}+2\lambda e^{a+(t+1)u}\,\mathrm{d} v_g,
\]
because with this choice of $\eta$ we have $\int_{\Omega}e^a\eta C\nabla_g\eta \, \mathrm{d} v_g=0.$ It follows that
\begin{equation}\label{40}
-\int_{\Omega}e^{a+2tu}|\nabla_g u|^{2}\,\mathrm{d} v_g \leq -\frac{\lambda}{t^2}\int_{\Omega} e^{a+(2t+1)u}\, \mathrm{d} v_g +\frac{2\lambda}{t^2}\int_{\Omega} e^{a+(t+1)u}\,\mathrm{d} v_g.
\end{equation}
Testing $\phi$ in the weak solution condition we obtain
\begin{equation}\label{41}
2t\int_{\Omega}e^{a+2tu}|\nabla_g u|^{2}\,\mathrm{d} v_g \leq\lambda\int_{\Omega}e^{a+(2t+1)u}\,\mathrm{d} v_g,
\end{equation}
because with this choice of $\phi$ we can check that $\int_{\Omega}e^a\phi C\cdot\nabla_gu \,\mathrm{d} v_g=0,$ because $\Omega$ is simply connected. Using \eqref{40} and \eqref{41} we have
\[
\left( \frac{1}{t^2}-\frac{1}{2t} \right)\int_{\Omega}e^{a+(2t+1)u}\,\mathrm{d} v_g \leq \frac{2}{t^2}\int_{\Omega} e^{a+(t+1)u}\,\mathrm{d} v_g.
\]
Using H\"{o}lder inequality with conjugate exponents $(2t+1)/(t+1)$ and $(2t+1)/t$, we have
\[
\left( \frac{1}{t^2}-\frac{1}{2t} \right)\int_{\Omega}e^{a+(2t+1)u}\,\mathrm{d} v_g \leq \frac{2C_1}{t^2} \left[\int_{\Omega}e^{a+(2t+1)u}\right]^{(t+1)/(2t+1)}\Vert\Omega\Vert^{t/(2t+1)}\,\mathrm{d} v_g,
\]
where $C_1=\displaystyle\left[\sup_{\overline{\Omega}}e^{t/(t+1)a}\right]^{(t+1)/(2t+1)}$. Therefore
\[
\parallel e^{1/(2t+1)a+u}\parallel_{L^{2t+1}}\leq \left[\frac{2t}{2t-t^2}\right]^{1/t}C_1\Vert\Omega\Vert^{1/(2t+1)}.
\]
\end{proof}

\begin{remark}
In the above estimate we used that $1/t^2-1/(2t)> 0$ which is an immediately consequence of our assumption $0<t<2$.
\end{remark}

\begin{remark}
The above estimate said that $e^{a+u}$ is bounded uniformly in $\lambda$ over $L^p(\Omega)$ for all $p<p_0:=4+1=5\, .$ By elliptic estimates, $\underline{u}_\lambda$ is uniformly bounded in $W_{0}^{1,p}(\Omega)$. Thus $u^*$ is a weak solution of $(P_{\lambda^*})$. Taking the limit in $\lambda$, we obtain the same $L^{p}$ estimate above to extremal solution $u^*$.
\end{remark}

\begin{proposition}\label{54}
If $1\leq N \leq 9$ then $u^*$ is a classical solution of $(P_{\lambda^*})$.
\end{proposition}

\begin{proof}
Note that $e^{a+u}\in L^{p}(\Omega)$ with $p<5$. By elliptic regularity we have $u^*\in W_{0}^{2,p}(\Omega)$ and by Sobolev immersion $u^*\in C^{0,\alpha}(\overline{\Omega})$ if $N<10$. Thus $u^*$ is a classical solution of $(P_{\lambda^*})$.
\end{proof}

\begin{lemma}\label{26}
If $u$ is a semi-stable solution of \eqref{01} with $0<\lambda<\lambda^*$, $f(u)=(1+u)^m$, $b>0$ and $\frac{m}{b}-\frac{\sqrt{m(m-1)}}{b}<t<\frac{m}{b}+\frac{\sqrt{m(m-1)}}{b}$, holds the following estimate
\[
\parallel e^{am/[2bt+m-1]}(1+u)^m\parallel_{L^{[2bt+m-1]/m}}\leq \left( 1-\frac{b^2t^2}{m[2bt-1]} \right)^{-1/[tb]}C_1\Vert\Omega\Vert^{1/[(2t+1)b]}.
\]
\end{lemma}

\begin{proof}
Define $\eta:=(1+u)^{bt}-1$ and $\phi:=(1+u)^{2bt-1}-1$ where $b>0$. Testing $\eta$ in the semistability condition we have
\[
\int_{\Omega}b^2t^2e^a(1+u)^{2bt-2}|\nabla_g u|^2 \,\mathrm{d} v_g \geq\lambda m\int_{\Omega} e^a(1+u)^{(m-1)}\left[(1+u)^{2bt} - 2(1+u)^{bt}\right]\,\mathrm{d} v_g,
\]
because with this choice of $\eta$ we have $\int_{\Omega}e^a\eta C\nabla_g\eta \, \mathrm{d} v_g=0.$ Testing $\phi$ in the weak solution condition we obtain
\[
(2bt-1)\int_{\Omega}e^a(1+u)^{2bt-2}|\nabla_g u|^2\,\mathrm{d} v_g \leq\lambda\int_{\Omega}e^a(1+u)^{2bt+m-1}\,\mathrm{d} v_g,
\]
because with this choice of $\phi$ we can check that $\int_{\Omega}e^a\phi C\nabla_gu \,\mathrm{d} v_g=0$. It follows that 
\[
\left( 1-\frac{b^2t^2}{m[2bt-1]} \right)\int_{\Omega}e^a(1+u)^{2bt+m-1}\, \mathrm{d} v_g \leq 2\int_{\Omega}e^a(1+u)^{bt+m-1}\,\mathrm{d} v_g.
\]
Using H\"{o}lder inequality with conjugate exponents $\frac{2bt+m-1}{bt+m-1}$ and $\frac{2bt+m-1}{bt}$, we have
\[
\left( 1-\frac{b^2t^2}{m[2bt-1]} \right)\int_{\Omega}e^a(1+u)^{2bt+m-1} \mathrm{d} v_g \leq C_1\left[\int_{\Omega}e^a(1+u)^{2bt+m-1}\right]^{\frac{bt+m-1}{2bt+m-1}}\Vert\Omega\Vert^{\frac{bt}{2bt+m-1}}\mathrm{d} v_g,
\]
where $C_1=2\displaystyle\left[\sup_{\overline{\Omega}}e^{at/(t+1)}\right]^{t/(2t+1)}$. Therefore
\[
\parallel e^{am/[2bt+m-1]}(1+u)^m\parallel_{L^{[2bt+m-1]/m}}\leq \left( 1-\frac{b^2t^2}{m[2bt-1]} \right)^{-1/[tb]}C_1\Vert\Omega\Vert^{1/[(2t+1)b]}.
\]
\end{proof}

\begin{remark}
The above estimate said that $e^a(1+u)^m$ is bounded uniformly in $\lambda$ over $L^p(\Omega)$ for all $p<3-\frac{1}{m}+\frac{2}{m}\sqrt{m(m-1)}$. By elliptic estimates, $\underline{u}_\lambda$ is uniformly bounded in $W_{0}^{1,p}(\Omega)$. Thus $u^*$ is a weak solution of $(P_{\lambda^*})$. Taking the limit in $\lambda$, we obtain the same $L^{p}$ estimate above to extremal solution $u^*$.
\end{remark}

\begin{proposition}\label{55}
If $1\leq N \leq 10$ then $u^*$ is a classical solution of $(P_{\lambda^*})$.
\end{proposition}

\begin{proof}
Since $e^a(1+u^*)^m\in L^{p}(\Omega)$ with $p<3-\frac{1}{m}+\frac{2}{m}\sqrt{m(m-1)}$, we can use elliptic regularity to obtain $u^*\in W_{0}^{2,p}(\Omega)$ and by Sobolev immersion $u^*$ is a classical solution if $N < 6+\frac{4}{m-1}\left(\sqrt{m(m-1)}+1\right)$. Observe that $\frac{\sqrt{m(m-1)}+1}{m-1}>1$ An immediate consequence is that if $m>1$ and $N\leq 10$, $u^*$ is a classical solution of $(P_{\lambda^*})$.
\end{proof}

\begin{proof}[Proof of Theorem \ref{13}]
It follows from Propositions \ref{51}, \ref{54}, \ref{55}
\end{proof}

\section{Symmetry and monotonicity}\label{63}
\begin{proof}[Proof of Theorem \ref{07}]
Let $u\in C^2(\mathcal{B}_R)$ a stable solution of \eqref{01}. The stability condition \eqref{02} is equivalent to the positivity of the first eigenvalue of $L_{u,\lambda}$ in $\mathcal{B}_R$, i.e.,
\[
\mu_{1,\lambda}=\displaystyle\inf_{\xi\in W_{0}^{1,2}(\mathcal{B}_R)\setminus\lbrace 0 \rbrace}\frac{\int_{\mathcal{B}_R}\lbrace|\nabla_g \xi|^2 + \xi A\cdot\nabla_g \xi -\lambda f^\prime (u)\xi^2\rbrace\,\mathrm{d} v_g}{\int_{\mathcal{B}_R}\xi^2\,\mathrm{d} v_g}> 0.
\]
Now, consider $u_\theta=\frac{\partial u}{\partial\theta}$ any angular derivative of $u$. By the fact $u\in C^2(\mathcal{B}_R)$, we have
\[
\int_{\mathcal{B}_R}|\nabla_g u_\theta|^2\,\mathrm{d} v_g<\infty.
\]
Moreover, the regularity up the boundary of $u$ and the fact that $u=0$ on $\partial\mathcal{B}_R$ give that $u_\theta=0$ on $\partial\mathcal{B}_R$. Hence, $u_\theta\in W_{0}^{1,2}(\mathcal{B}_R)$. Differentiate the problem \eqref{01} we obtain that $u_\theta$ weakly satisfies
\[
\left\{
\begin{alignedat}{3}
-\Delta_gu_\theta + A\cdot\nabla_g u_\theta= & \, \lambda f^\prime (u)u_\theta & \quad \text{in} & \quad\mathcal{B}_R, \\
u_\theta = & \, 0  &  \text{on} & \quad \partial\mathcal{B}_R.\\
\end{alignedat}
\right.
\]
Multiplying the above equation by $u_\theta$ and integrating by parts we have
\[
\int_{\mathcal{B}_R}\lbrace|\nabla_g u_\theta|^2 + u_\theta A\cdot\nabla_g u_\theta-\lambda f^\prime (u)u_\theta^2\rbrace\,\mathrm{d} v_g=0.
\]
It follows that either $|u_\theta|$ is the first eigenvalue of linearized operator at $u$ or $u_\theta=0$. But the first alternative can not occur because $\mu_{1,\lambda}>0$. It follows that $u_\theta=0$ for all $\theta\in\mathbb{S}^{N-1}$. Thus $u$ is radial. On the other hand, in spherical coordinates given by (\ref{08}), the Laplacian operator of $u=u(r,\theta_1,...,\theta_{N-1})$ is given by
\[
\Delta_g u=\frac{1}{\psi^{N-1}}(\psi^{N-1}u_r)_r+\frac{1}{\psi^2}\Delta_{\mathbb{S}^{N-1}}u,
\]
where $\Delta_{\mathbb{S}^{N-1}}$ is the Laplacian on the unit sphere $\mathbb{S}^{N-1}$. To prove the monotonicity, note that  since $u=u(r)$ and $A=A(r)$, the equation \eqref{01} becomes
\[
\int_{0}^{s}\int_{0}^{2\pi}e^a(\psi^{N-1}u_r)_r\,\mathrm{d} r\,\mathrm{d} \theta=\int_{0}^{s}\int_{0}^{2\pi}-e^a\psi^{N-1}f(u)\,\mathrm{d} r\,\mathrm{d} \theta.
\]
Therefore, $u_r < 0$.
\end{proof}

\subsection{Regularity in radial case}
In view of the previous section, we can write the problem \eqref{01} for radial solutions $u\in C^{2}(\mathcal{B}_R)$ as

\begin{equation}\label{04}
\left\{
\begin{alignedat}{3}
-(e^a\psi^{N-1}u_r)_r+e^a\psi^{N-1}C(r) u_r =&  \,\lambda e^a\psi^{N-1}f(u) & \quad \text{em} & \quad (0,R), \\
u > &\, 0 & \quad \text{em} & \quad (0,R),\\
u_r(0) = u(R) = & \, 0.
\end{alignedat}
\right.\tag{$R_{\lambda}$}
\end{equation}
In the same way, the semistability and weak solution condition becomes, respectively
\[
\int_{0}^{R}e^a\psi^{N-1}\xi_{r}^{2} + e^a\psi^{N-1}C(r)\xi \xi_r \,\mathrm{d} r \geq \int_{0}^{R}\lambda e^a\psi^{N-1}f^\prime (u)\xi^2\,\mathrm{d} r
\]
and
\[
\int_{0}^{R}e^a\psi^{N-1}u_r\phi_r + e^a\psi^{N-1}C(r) u_r\phi \,\mathrm{d} r=\int_{0}^{R}\lambda e^a\psi^{N-1}f(u)\phi\,\mathrm{d} r.
\]

In radial case, we obtain a more precise information about the $L^\infty$ norm of the extremal solution. Again, we will start with MEMS case.

\begin{lemma}\label{18}
If $u$ is a classical semi-stable solution of \eqref{04} with $f(u)=1/(1-u)^2$, then for all $0 < t < 2 + \sqrt{6}$ we have
\[
\Vert e^{2a/(2t+3)}\psi^{2(N-1)/(2t+3)}(1-u)^{-2} \Vert_{L^{t+3/2}}\leq \left[ \frac{4(2t+1)}{2+4t-t^2} \right]^{2/t}C_2R^{2/(2t+3)}.
\]
\end{lemma}

\begin{proof}
We follow the proof of Lemma \ref{03}. Let $0<t<2+\sqrt{6}$ and $u$ semi-stable classical solution of \eqref{04}. Define $\eta:=(1-u)^{-t}-1$ and $\phi:=(1-u)^{-2t-1}-1$. Applying $\eta$ in the semistablity condition we have
\begin{equation}\label{11}
-\int_{0}^{R}e^a\psi^{N-1}(1-u)^{-2t-2}u_{r}^{2}\,\mathrm{d} r \leq \frac{2}{t^2}\int_{0}^{R}e^a\psi^{N-1}(1-u)^{-2t-3}\,\mathrm{d} r+\frac{4}{t^2}\int_{0}^{R}e^a\psi^{N-1}(1-u)^{-t-3}\,\mathrm{d} r.
\end{equation}
Applying $\phi$ in the weak solution condition, it follows that
\begin{equation}\label{14}
\int_{0}^{R}e^a\psi^{N-1}u_{r}^{2}(2t+1)(1-u)^{-2t-2}\,\mathrm{d} r \leq\int_{0}^{R}e^a\psi^{N-1}(1-u)^{-2t-3}\,\mathrm{d} r.
\end{equation}
Using \eqref{11} and \eqref{14} we obtain
\[
\left( \frac{2}{t^2}-\frac{1}{2t+1} \right)\int_{0}^{R}e^a\psi^{N-1}(1-u)^{-2t-3}\,\mathrm{d} r \leq \frac{4}{t^2}\int_{0}^{R}e^a\psi^{N-1}(1-u)^{-t-3}\,\mathrm{d} r.
\]
Using H\"{o}lder inequality with conjugate exponents $(2t+3)/(t+3)$ and $(2t+3)/t$,
\[
\begin{alignedat}{2}
\left( \frac{2}{t^2}-\frac{1}{2t+1} \right)\int_{0}^{R}&e^a\psi^{N-1}(1-u)^{-2t-3}\,\mathrm{d} r \\
& \leq \frac{4}{t^2}R^{t/(2t+3)} C_2(t,\psi) \left[\int_{0}^{R}e^a\psi^{N-1}(1-u)^{-2t-3}\,\mathrm{d} r\right]^{(t+3)/(2t+3)},
\end{alignedat}
\]
where $C_2:=\displaystyle\left[\sup_{[0,R]} e^{at/(t+3)}\psi^{(N-1)t/(t+3)}\right]^{(t+3)/(2t+3)}$. Thus,
\[
\left(\frac{2}{t^2}-\frac{1}{2t+1}\right)\left[\int_{0}^{R}e^a\psi^{N-1}(1-u)^{-2t-3}\right]^{t/(2t+3)}\leq \frac{4}{t^2}C_2R^{t/(2t+3)}
\]
and therefore
\[
\parallel e^{2a/(2t+3)}\psi^{2(N-1)/(2t+3)}(1-u)^{-2}\parallel_{L^{t+3/2}}\leq \left[\frac{4(2t+1)}{2+4t-t^2}\right]^{2/t}C_2R^{2/(2t+3)}.
\]
\end{proof}

\begin{lemma}\label{17}
Let $u$ be a radially decreasing and semi-stable classical solution of \eqref{04} with $f(u)=1/(1-u)^2$. If $1 \leq p < \infty$, we have the estimate
\[
u(0)\geq u(r)\geq u(0)-rC_3\Vert e^{a/p}\psi^{(N-1)/p}(1-u)^{-2} \Vert_p .
\]
\end{lemma}

\begin{proof}
By the Mean value theorem, there exists $c\in(0,r)$ such that
\begin{equation}\label{145}
-u(r)+u(0)=-u^\prime(c)r
\end{equation}
Integrating the equation \eqref{04} from $0$ to $c$ we obtain
\[
\begin{alignedat}{2}
-\mathrm{e}^{a(c)}\psi^{N-1}(c)&u^\prime(c) = \int_{0}^{c}\mathrm{e}^a\psi^{N-1}(1-u)^{-2}\\
                            & \leq \left[\int_{0}^{R}\mathrm{e}^a\psi^{(N-1)}(1-u)^{-2p}\right]^{1/p}\left[\int_{0}^{R}\mathrm{e}^{a(1-1/p)}\psi^{(N-1)(1-1/p)}\mathrm{d} r\right]^{p/(p-1)}.
\end{alignedat}
\]
Using \eqref{145} we conclude the proof because
\[
-u(r)+u(0)\leq rC_3\parallel \mathrm{e}^{a/p}\psi^{(N-1)/p}(1-u)^{-2}\parallel_{L^p}
\]
where $C_3=\mathrm{e}^{-a(c)}\psi^{1-N}(c)\left[\int_{0}^{R}\mathrm{e}^{a(1-1/p)}\psi^{(N-1)(1-1/p)}\mathrm{d} r\right]^{p/(p-1)}$.
\end{proof}

\begin{lemma}\label{05}
Let $u$ a radially decreasing and semi-stable classical solution of \eqref{04} with $1 \leq N\leq 7$. Then, for all $0<t<2+\sqrt{6}$, we have
\[
\int_{0}^{r}\frac{e^a\psi^{N-1}}{D_2(r)^{2t+3}}\, \mathrm{d} r\leq \left(\frac{4(2t+1)}{4t+2-t^2}\right)^{(2t+3)/t},
\]
where $D_2(r):=1-\Vert u \Vert_\infty + C_4\left(4(2t+1)/(2+4t-t^2)\right)^{2/t}R^{1/p}r.$
\end{lemma}

\begin{proof}
Take $p=t+3/2$. By Lemma \ref{17},
\[
1-u(r)\leq 1-u(0)+rC_3\parallel e^{a/p}\psi^{(N-1)/p}(1-u)^{-2}\parallel_{L^p}.
\]
Multiplying some positive terms and using Lemma \ref{18}, it follows that
\[
e^{-a}\psi^{-(N-1)}(1-u(r))^{2t+3}\leq e^{-a}\psi^{-(N-1)}\left(1-u(0)+C_1C_3\left( \frac{4(2t+1)}{2+4t-t^2} \right)^{2/t}R^{1/p}r\right)^{2t+3}.
\]
We have
\[
\int_{0}^{r}\frac{e^a\psi^{N-1}dr}{D_2(r)^{2t+3}}\leq \int_{0}^{r}\frac{e^a\psi^{N-1}dr}{\left(1-u(r)\right)^{2t+3}}.
\]
where $C_4:=C_1C_3$. Thus, 
\[
\int_{0}^{r}\frac{e^a\psi^{N-1}dr}{D_2(r)^{2t+3}}\leq \left( \frac{4(2t+1)}{2+4t-t^2} \right)^{(2t+3)/t}R^{1/p}.
\]
\end{proof}

We split the proof of Theorem \ref{06} in three cases, namely, MEMS, Gelfand and Power cases.

\begin{proof}[Proof of Theorem \ref{06} (MEMS case)]
Using the Lemma \ref{05}, we have
\begin{equation}\label{19}
\int_{0}^{r}\frac{e^a\psi^{N-1}}{D_2(r)^{2t+3}} \,\mathrm{d} r\leq \left(\frac{4(2t+1)}{4t+2-t^2}\right)^{(2t+3)/t}R^{1/p}.
\end{equation}
Calculating the left-hand side above, we have
\begin{equation}\label{20}
\begin{alignedat}{2}
\int_{0}^{r}\frac{e^a\psi^{N-1}}{D_2(r)^{2t+3}} \,\mathrm{d} r&=\frac{\int_{0}^{r}e^a\psi^{N-1}\,\mathrm{d} r}{D_2(r)^{2t+3}}+(2t+4)\int_{0}^{r}\frac{D_{2}^{\prime}\int_{0}^{r}e^a\psi^{N-1}\,\mathrm{d} r}{D_{2}^{2t+4}}\, \mathrm{d} r\\
&\geq \frac{\int_{0}^{r}e^a\psi^{N-1}\,\mathrm{d} r}{D_2(r)^{2t+3}}.
\end{alignedat}
\end{equation}
Applying \eqref{20} in \eqref{19}, it follows that
\begin{equation}\label{21}
\frac{\int_{0}^{r}e^a\psi^{N-1}\,\mathrm{d} r}{D_2(r)^{2t+3}}\leq\left(\frac{4(2t+1)}{4t+2-t^2}\right)^{(2t+3)/t}R^{1/p}.
\end{equation}
Calculating the equation \eqref{21} and taking $\lambda\nearrow\lambda^*$ we have
\[
\Vert u^* \Vert_\infty \leq 1 - C,
\]
where
\[
\begin{alignedat}{2}
C &:=\left[\int_{0}^{R}e^a\psi^{N-1}\,\mathrm{d} r\right]^{1/(2t+3)}\left(4(2t+1)/(4t+2-t^2)\right)^{-1/t}R^{1/(p^2)}\\
&-C_4\left(4(2t+1)/(2+4t-t^2)\right)^{2/t}R^{1+1/p}.
\end{alignedat}
\]
\end{proof}

With a slight variation of the above arguments, the same approach works for the Gelfand problem with advection.

\begin{lemma}\label{45}
Let $u$ a radially decreasing and semi-stable classical solution of \eqref{04} with $f(u)=e^u$. If $1 \leq p < \infty$, we have the estimate
\[
u(0)\geq u(r)\geq u(0)-rC_3\Vert e^{a/p+u}\psi^{(N-1)/p}\Vert_p.
\]
\end{lemma}

\begin{proof}
There exists $c\in(0,r)$ such that
\begin{equation}\label{146}
-u(r)+u(0)=-u^\prime(c)r.
\end{equation}
Integrating the equation \eqref{04} from $0$ to $c$ we obtain
\[
\begin{alignedat}{2}
-\mathrm{e}^{a(c)}\psi^{N-1}(c)u^\prime(c)&=\int_{0}^{c}\mathrm{e}^{a+u}\psi^{N-1}\\
                                          & \leq \left[\int_{0}^{R}\mathrm{e}^{a+up}\psi^{N-1}\right]^{1/p}\left[\int_{0}^{R}\mathrm{e}^{a(1-1/p)}\psi^{(N-1)(1-1/p)}\mathrm{d} r\right]^{p/(p-1)}.
\end{alignedat}
\]
Using \eqref{146} we obtain
\[
-u(r)+u(0)\leq C_3\parallel \mathrm{e}^{a/p+u}\psi^{(N-1)/p}\parallel_{L^p}r
\]
where $C_3=e^{-a(c)}\psi^{1-N}(c)\left[\int_{0}^{R}\mathrm{e}^{a(1-1/p)}\psi^{(N-1)(1-1/p)}\mathrm{d} r\right]^{p/(p-1)}$.
\end{proof}

\begin{proof}[Proof of Theorem \ref{06} (Gelfand case)]
Take $p=2t+1$. By Lemma \ref{45} and using Lemma \ref{39}, it follows that
\[
-u(r)\leq -u(0)+C_1C_3\left((2t)/(2t-t^2) \right)^{1/t}R^{1/p}r.
\]
Multiplying some positive terms we have
\[
e^{a}\psi^{(N-1)}e^{-u(2t+1)}\leq e^{a}\psi^{(N-1)}e^{(-u(0)+C_1C_3\left(4(2t)/(2t-t^2) \right)^{1/t}R^{1/p}r)(2t+1)}.
\]
Thus, 
\[
\int_{0}^{r}e^{a}\psi^{(N-1)}e^{(u(0)-C_4\left(4(2t)/(2t-t^2) \right)^{1/t}R^{1/p}r)(2t+1)} \leq C_1\left( \frac{2t}{2t-t^2} \right)^{(2t+1)/t}R^{1/p},
\]
where $C_4:=C_1C_3$. Calculating the left-hand side above, we have
\[
e^{(u(0)-C_4\left(4(2t)/(2t-t^2) \right)^{1/t}R^{1/p}r)(2t+1)}\int_{0}^{r}e^{a}\psi^{(N-1)}\,\mathrm{d} r\leq C_4\left( \frac{2t}{2t-t^2} \right)^{(2t+1)/t}R^{1/p}.
\]
Taking the limit $\lambda\nearrow\lambda^*$ we have
\[
\parallel u^*\parallel_\infty\leq \frac{\ln\left(C_1R^{1/p}\left( \frac{2t}{2t-t^2} \right)^{(2t+1)/t}\right)}{(2t+1)}+C_4\left(4(2t)/(2t-t^2) \right)^{1/t}R^{1+1/p}.
\] 
\end{proof}

\begin{lemma}\label{58}
Let $u$ a radially decreasing and semi-stable classical solution of \eqref{04} with $f(u)=(1+u)^m$. If $1 \leq p < \infty$, we have the estimate
\[
u(0)\geq u(r)\geq u(0)-rC_3\parallel e^{a/p}\psi^{(N-1)/p}(1+u)^m\parallel_{L^p}.
\]
\end{lemma}

\begin{proof}
There exists $c\in(0,r)$ such that
\begin{equation}\label{147}
-u(r)+u(0)=-u^\prime(c)r.
\end{equation}
Integrating the equation \eqref{04} from $0$ to $c$ we obtain
\[
\begin{alignedat}{2}
-\mathrm{e}^{a(c)}\psi^{N-1}(c)u^\prime(c)&=\int_{0}^{r}e^a\psi^{N-1}(1+u)^m\leq\\
& \left[\int_{0}^{r}e^{a}\psi^{(N-1)}(1+u)^{mp}\,\mathrm{d} r\right]^{1/p}\left[\int_{0}^{r}e^{a(1-1/p)p^\prime}\psi^{(N-1)(1-1/p)p^\prime}\,\mathrm{d} r\right]^{1/p^\prime}.
\end{alignedat}
\]
Using \eqref{147} we obtain
\[
-u(r)+u(0)\leq C_3\parallel \mathrm{e}^{a/p+u}\psi^{(N-1)/p}\parallel_{L^p}r
\]
where $C_3=e^{-a(c)}\psi^{1-N}(c)\left[\int_{0}^{R}\mathrm{e}^{a(1-1/p)}\psi^{(N-1)(1-1/p)}\mathrm{d} r\right]^{p/(p-1)}$.
\end{proof}
Integrating the equation \eqref{04} from $0$ to $r$ we obtain
\[
\begin{alignedat}{2}
-e^a\psi^{N-1}u^\prime(r)&=\int_{0}^{r}e^a\psi^{N-1}(1+u)^m\leq\\
& \left[\int_{0}^{r}e^{a}\psi^{(N-1)}(1+u)^{mp}\,\mathrm{d} r\right]^{1/p}\left[\int_{0}^{r}e^{a(1-1/p)p^\prime}\psi^{(N-1)(1-1/p)p^\prime}\,\mathrm{d} r\right]^{1/p^\prime},
\end{alignedat}
\]
Integrating again the last inequality from $0$ to $r$ we conclude the proof because
\[
-u(r)+u(0)\leq C_3\parallel e^{a/p}\psi^{(N-1)/p}(1+u)^m\parallel_{L^p}r.
\]
where $C_3=e^{-a(c)}\psi^{1-N}(c)\left[\int_{0}^{R}e^{a(1-1/p)}\psi^{(N-1)(1-1/p)}\,\mathrm{d} r\right]^{1/p^\prime}$.

\begin{proof}[Proof of Theorem \ref{06} (Power case)]
Take $p=(2bt+m-1)/m$. By Lemma \ref{58} and using Lemma \ref{26}, it follows that
\[
\begin{alignedat}{2}
\int_{0}^{r}&e^{a/p}\psi^{(N-1)/p}(1+u(0))^m  \leq\\
 &\int_{0}^{r}e^{a/p}\psi^{(N-1)/p}(1+u(r) + C_1C_3\left( 1-\frac{b^2t^2}{m[2bt-1]} \right)^{-1/[tb]}R^{1/(2bt+b)}r)^m\,\mathrm{d} r.
\end{alignedat}
\]
Thus, we have
\[
\Vert u \Vert^m \leq 2^mC_1C_3\left( 1-\frac{b^2t^2}{m[2bt-1]} \right)^{-\frac{1}{tb}}R^{\frac{1}{2bt+b}}+\frac{2^mC_{1}^{m}C_{3}^{m}}{m+1}\left( 1-\frac{b^2t^2}{m[2bt-1]} \right)^{-\frac{m}{tb}}R^{\frac{m}{2bt+b}+m+1}.
\]
Using the above inequality and taking the limit $\lambda\nearrow\lambda^*$, it follows that
\[
\Vert u^* \Vert^m \leq 2^mC_1C_3\left( 1-\frac{b^2t^2}{m[2bt-1]} \right)^{-\frac{1}{tb}}R^{\frac{1}{2bt+b}}+\frac{2^mC_{1}^{m}C_{3}^{m}}{m+1}\left( 1-\frac{b^2t^2}{m[2bt-1]} \right)^{-\frac{m}{tb}}R^{\frac{m}{2bt+b}+m+1}
\]
\end{proof}

\section{Existence of nonminimal solutions}\label{66}
\begin{lemma}\label{31}
Let $u$ and $v$ a weak solution and a weak supersolution, respectively, of \eqref{01}.
\begin{itemize}
\item[(i)]If $\mu_1(\lambda, u)>0$, then $u\leq v$ a.e. in $\Omega$.
\item[(ii)]If $u$ is a regular solution and if $\mu_1(\lambda, u)=0$, then $u=v$ a.e. in $\Omega.$
\end{itemize}
\end{lemma}

\begin{proof}
Let $\theta\in[0,1]$ and $0\leq \phi\in W_{0}^{1,2}(\Omega)$. By convexity of $s\rightarrow f(s)$ we have
\[
\begin{alignedat}{2}
I_{\theta,\phi}&:=\int_{\Omega}\left(\nabla_g(\theta u+(1-\theta)v)\cdot\nabla_g\phi+\phi A\cdot\nabla_g(\theta u+(1-\theta)v)\right)\,\mathrm{d} v_g\\&-\int_{\Omega}\lambda f(\theta u+(1-\theta)v)\phi\,\mathrm{d} v_g\\
&\geq\lambda\int_{\Omega}\left(\theta f(u)+(1-\theta)f(v)-f(\theta u - (1-\theta)v)\right)\phi\, \mathrm{d} v_g\geq 0.
\end{alignedat}
\]
Since $I_{1,\phi}=0$, the derivative of $I_{\theta,\phi}$ at $\theta=1$ is nonpositive. If $\mu_1(\lambda, u)>0$, clearly $u\leq v$. We shall prove that this holds true if $\mu_1(\lambda, u) \geq 0$. In deed, we have
\begin{equation}\label{30}
\int_{\Omega}\left(\nabla_g(u-v)\cdot\nabla_g\phi+\phi A\cdot\nabla_g(u-v)\right)\, \mathrm{d} v_g-\int_{\Omega}\lambda f^\prime(u)(u-v)\phi\, \mathrm{d} v_g= 0.
\end{equation}
Since $I_{\theta,\phi}\geq 0$ for any $\theta\in[0,1]$ and $I_{1,\phi}=\partial_\theta I_{1,\phi}=0,$ we have
\[
\partial_{\theta\theta}^{2}I_{1,\phi}=-\int_{\Omega}\lambda f^{\prime\prime}(u)(u-v)^2\phi\, \mathrm{d} v_g\geq 0.
\]
Take $\phi=(u-v)^+$. We have $(u-v)^+=0$ in $\Omega$ and we get $\int_{\Omega}|\nabla_g(u-v)^+|^2\, \mathrm{d} v_g=0$. It follows that $u\leq v$ a.e. in $\Omega$ as claimed. Now, if $\mu_{1,\lambda}(u)=0$ let $\psi_{1,\lambda}$ the first eigenvalue of $L_{u,\lambda}$. Observe that $\psi_{1,\lambda}$ is in the kernel of the linearized operator $L_{u,\lambda}$, and \eqref{30} is valid if we replace $u-v$ with $u-v-t\psi_{1,\lambda}$. We have
\[
\begin{alignedat}{2}
\int_{\Omega}&\left(|\nabla_g(u-v-t\psi_{1,\lambda})^+|^2+(u-v-t\psi_{1,\lambda})^+A\cdot\nabla_g(u-v-t\psi_{1,\lambda})^+\right)\, \mathrm{d} v_g\\
&-\int_{\Omega}\lambda f(u)((u-v-t\psi_{1,\lambda})^+)^2\, \mathrm{d} v_g= 0.
\end{alignedat}
\]
We claim that if $u<v-\overline{t}\phi_{1,\lambda}$ on a set $\Omega^\prime$ of positive measure, then there exists $\epsilon>0$ such that $u<v-t\phi_{1,\lambda}$ a.e. in $\Omega$ for any $\overline{t}\leq t <\overline{t}+\epsilon.$ Since we have a variational characterization of $\phi_{1,\lambda}$ we get that $(u-v-t\phi_{1,\lambda})^+=\beta\phi_{1,\lambda}$ a.e. in $\Omega$ for some $\beta\in \mathbb{R}.$ We can find, by assumption, a set $\Omega^\prime\subset \Omega$ of positive measure such that $u<v-\overline{t}\phi_{1,\lambda}-\delta$ for $\delta>0$ and consequently, for some $\epsilon>0$ sufficient small that $u<v-t\phi_{1,\lambda}$ in $\Omega^\prime$ for any $\overline{t}\leq t \leq \overline{t}+\epsilon.$ Hence $\beta\phi_{1,\lambda}=0$ a.e. in $\Omega^\prime$. Since $\phi_{1,\lambda}>0$ in $\Omega$ we have $\beta=0$ and $u<v+t\phi_{1,\lambda}$ a.e. in $\Omega$ for any $\overline{t}\leq t \leq \overline{t}+\epsilon$ and this finishes the proof of claim. Now, by contradiction, assume that $u$ is not equal to $v$ a.e. in $\Omega$. Since $u\leq v,$ we find a set $\Omega^\prime$ of positive measure so that $u<v$ in $\Omega^\prime$. Applying the above claim with $\overline{t}=0$ we get some $\epsilon>0$, $u<v-t\phi_{1,\lambda}$ a.e. in $\Omega$ for any $0\leq t<\epsilon.$ Set now $t_0=\sup\lbrace t>0:u<v-t\phi_{1,\lambda}\text{ a.e. in }\Omega\rbrace.$ Clearly $u\leq v-t_0\phi_{1,\lambda}$ a.e. in $\Omega$. The claim and maximal property of $t_0$ imply that necessarily $u=v-t_0\phi_{1,\lambda}$ a.e. in $\Omega$ since \eqref{30} holds for any $0\leq\phi\in W_{0}^{1,2}(\Omega).$ Taking $\phi=v-u$ and arguing as before we have $\int_{\Omega}|\nabla_g(u-v)|^2\, \mathrm{d} v_g=0$ contradicting the assumption that $u<v$ on a set of positive measure.
\end{proof}

\begin{proof}[Proof of Theorem \ref{36}]
Using Theorem \ref{13}, we have that $u^*$ exists as a classical solution. On the other hand, we have that $\mu_{1,\lambda^*}\geq 0.$ If we suppose that $\mu_{1,\lambda^*}>0$, then the Implicit Function Theorem could be applied to the operator $L_{u^*,\lambda^*}$ to allow for the continuation of the minimal branch $\lambda\nearrow \underline{u}_\lambda$ beyond $\lambda^*$,  which is a contradiction. Therefore $\mu_{1,\lambda^*}=0$. The uniqueness of $u^*$ in the class of weak solutions follows from the Lemma \ref{31}.
\end{proof}

\begin{proposition}
If $0<\lambda<\lambda^*$, the minimal solutions are stable.
\end{proposition}

\begin{proof}
Define
\[
\lambda^{**}=\sup\lbrace\lambda > 0 : \underline{u}_\lambda\text{ is a stable solution for \eqref{01}}\rbrace.
\]
Obviously $\lambda^{**}$ satisfies $\lambda^{**}\leq\lambda^*.$ If $\lambda^{**} < \lambda^*,$ then $\underline{u}_{\lambda^{**}}$ is a minimal solution of $(P_{\lambda^{**}})$. For $\lambda\leq\lambda^{**}$, we have that $\lim_{\lambda\nearrow\lambda^{**}}\underline{u}_\lambda\leq \underline{u}_{\lambda^{**}}$. Since $u^{**}$ is solution of $(P_{\lambda^{**}})$ and by minimality follows that $\lim_{\lambda\nearrow\lambda^{**}}\underline{u}_\lambda = \underline{u}_{\lambda^{**}}$ and $\mu_{1,\lambda^{**}}\geq 0$. If we suppose that $\mu_{1,\lambda^{**}}=0,$ we get that $\underline{u}_{\lambda^{**}}=\underline{u}_\lambda$ for any $\lambda^{**}< \lambda <\lambda^*.$ But this is a contradiction, which proves that $\lambda^{**}=\lambda^*.$
\end{proof}

\begin{proposition}
For each $x\in \Omega$, the function $\lambda\rightarrow \underline{u}_\lambda(x)$ is differentiable and strictly increasing on $(0,\lambda^*).$
\end{proposition}

\begin{proof}
Since $\underline{u}_\lambda$ is stable, the linearized operator $L_{\underline{u}_\lambda,\lambda}$ at $u_\lambda$ is invertible for any $0< \lambda < \lambda^*$. By the Implicit Function Theorem $\lambda\rightarrow \underline{u}_\lambda(x)$ is differentiable in $\lambda$. By monotonicity, $\frac{\mathrm{d} \underline{u}_\lambda}{\mathrm{d} \lambda}(x)\geq 0$ for all $x\in \Omega.$ Finally, by differentiating \eqref{01} with respect to $\lambda$ we get that $\frac{\mathrm{d} \underline{u}_\lambda}{\mathrm{d} \lambda}(x)> 0$, for all $x\in\Omega$.
\end{proof}

It is standard to show the existence of a second branch of solutions near $\lambda^*.$ We make use of Mountain Pass Theorem to provide a variational characterization for this solutions. To apply the Mountain Pass Theorem we will need to truncate the singular nonlinearity into a subcritical case, that is, we consider a regularized $C^1$ nonlinearity $g_\epsilon(u),$ $0<\epsilon<1$ of the following form for MEMS case
\[
g_\epsilon(u)=\left\{
\begin{alignedat}{3}
\frac{1}{(1-u)^2}\hspace{3,2cm} & \text{ if } & u< 1-\epsilon\\
\frac{1}{\epsilon^2}-\frac{2(1-\epsilon)}{p\epsilon^3}+\frac{2u^p}{p\epsilon^3(1-\epsilon)^{p-1}}& \text{ if } & u\geq 1-\epsilon
\end{alignedat}
\right.
\]
and for Gelfand or Power-type
\[
g_\epsilon(u)=\left\{
\begin{alignedat}{3}
f(u)\hspace{6,2cm} & \text{ if } & u< t_0-\epsilon\\
f(s_0-\epsilon)-\frac{f^\prime(s_0-\epsilon)(s_0-\epsilon)}{p}+\frac{f^\prime(s_0-\epsilon)u^p}{p(s_0-\epsilon)^{p-1}}& \text{ if } & u\geq t_0-\epsilon
\end{alignedat}
\right.
\]
where $p>1$ if $N=1,2$ and $1<p<(N+2)/(N-2)$ if $3\leq N\leq N^*.$ For $\lambda\in (0,\lambda^*)$ and $A=\nabla_ga$ we associate the elliptic problem
\begin{equation}\label{32}
\left\{
\begin{alignedat}{3}
-\mathrm{div}\left( e^{-a}\nabla_g u \right)= & \, \lambda e^{-a}g_\epsilon(u) & \quad \text{in} & \quad\Omega, \\
u = & \, 0  &  \text{on} & \quad \partial\Omega,\\
\end{alignedat}
\right.\tag{$S_{\lambda}$}
\end{equation}
We can define a energy functional on $W_{0}^{1,2}(\Omega)$ associated to \eqref{32} given by
$$J_{\epsilon,\lambda}(u)=\frac{1}{2}\int_{\Omega}e^{-a}|\nabla_gu|^2\,\mathrm{d} v_g-\lambda\int_{\Omega} e^{-a}G_\epsilon(u)\,\mathrm{d} v_g,$$
where $G_\epsilon(u)=\int_{-\infty}^{u}g_\epsilon(s)\,\mathrm{d} s.$ We can fix $0<\epsilon<\frac{1-\Vert u^* \Vert_\infty}{2}$ for MEMS case or $0<\epsilon<\frac{t_0-\Vert u^* \Vert_\infty}{2}$ for Gelfand and Power-type, and observe that for $\lambda$ close enough to $\lambda^*,$ the minimal solution $\underline{u}_\lambda$ of \eqref{01} is also a solution of \eqref{32} that satisfies $\mu_{1,\lambda}(-\mathrm{div}(e^{-a}\nabla_g)-\lambda g_{\epsilon}^{\prime}(\underline{u}_\lambda))>0.$

\begin{lemma}
If $1\leq N < N^*$ and if $\lambda$ is close enough to $\lambda^*$, then the minimal solution $\underline{u}_\lambda$ of \eqref{32} is a strict local minimum of $J_{\epsilon,\lambda}$ on $W_{0}^{1,2}(\Omega).$
\end{lemma}
\begin{proof}
Since $\mu_{1,\lambda}((-\mathrm{div}(e^{-a}\nabla_g)-\lambda g_{\epsilon}^{\prime}(\underline{u}_\lambda))>0$ and $\underline{u}_\lambda<1-\epsilon$, we have the inequality
$$\int_{\Omega}e^{-a}|\nabla_g \phi|^2\, \mathrm{d} v_g-2\lambda\int_{\Omega}\frac{e^{-a}\phi^2}{(1-\underline{u}_\lambda)^3}\, \mathrm{d} v_g\geq \mu_{1,\lambda}\int_{\Omega}\phi^2\, \mathrm{d} v_g,$$
for any $\phi\in W_{0}^{1,2}(\Omega)$. Now take $\phi\in W_{0}^{1,2}(\Omega)\cap C^1(\overline{\Omega})$ such that $\underline{u}_\lambda+\phi\leq 1-\epsilon$ and $\Vert \phi \Vert_{C^1}\leq\delta_\lambda$. Thus we have
\[
\begin{alignedat}{2}
J_{\epsilon,\lambda}&(\underline{u}_\lambda+\phi)-J_{\epsilon,\lambda}(\underline{u}_\lambda)\\
&=\frac{1}{2}\int_{\Omega}e^{-a}|\nabla_g\phi|^2\, \mathrm{d} v_g+\int_{\Omega}e^{-a}\nabla_g\underline{u}_\lambda\cdot\nabla_g\phi\, \mathrm{d} v_g - \lambda\int_{\Omega}e^{-a}\left( \frac{1}{1-\underline{u}_\lambda-\phi}-\frac{1}{1-\underline{u}_\lambda} \right)\, \mathrm{d} v_g\\
&\geq\frac{\mu_{1,\lambda}}{2}\int_{\Omega}\phi^2\, \mathrm{d} v_g-\lambda \Vert e^{-a} \Vert_\infty \int_{\Omega}\left( \frac{1}{1-\underline{u}_\lambda-\phi}-\frac{1}{1-\underline{u}_\lambda}-\frac{\phi}{(1-\underline{u}_\lambda)^2}-\frac{\phi^2}{(1-\underline{u}_\lambda)^3} \right)\, \mathrm{d} v_g.
\end{alignedat}
\]
For some $C>0$ we have
\[
\left\vert \frac{1}{1-\underline{u}_\lambda-\phi}-\frac{1}{1-\underline{u}_\lambda}-\frac{\phi}{(1-\underline{u}_\lambda)^2}-\frac{\phi^2}{(1-\underline{u}_\lambda)^3} \right\vert\leq C |\phi|^3
\]
and this implies
\[
J_{\epsilon,\lambda}(u_\lambda+\phi)-J_{\epsilon,\lambda}(u_\lambda)\geq\left( \frac{\mu_{1,\lambda}}{2}-C\lambda\Vert e^{-a} \Vert_\infty\delta_\lambda \right)\int_{\Omega}\phi^2\, \mathrm{d} v_g>0
\]
provided $\delta_\lambda$ is small enough. This proves that $\underline{u}_\lambda$ is a local minimum of $J_{\epsilon,\lambda}$ in the $C^1$ topology. We can apply Theorem 2.1 of \cite{KHAMOT2012} and get that $u_\lambda$ is a local minimum of $J_{\epsilon,\lambda}$ in $W_{0}^{1,2}(\Omega)$. For Gelfand and Power cases we take $\phi\in W_{0}^{1,2}(\Omega)\cap C^1(\overline{\Omega})$ such that $\underline{u}_\lambda+\phi\leq t_0-\epsilon$ and $\Vert \phi \Vert_{C^1}\leq\delta_\lambda$. With similar arguments we conclude that $u_\lambda$ is a local minimum of $J_{\epsilon,\lambda}$ in $W_{0}^{1,2}(\Omega)$.
\end{proof}

Now we proof the existence of a second solution for \eqref{32}. We need a version of mountain pass theorem \cite{AMBRAB1973}.

\begin{theorem}[Critical point of Mountain pass type]\label{33} Let $J$ be a $C^1$ functional defined on a Banach space $E$ that satisfies the Palais-Smale condition, that is, any sequence in $E$ such that $(J(u_n))_n$ is bounded and $J^\prime(u_n)\rightarrow 0$ in $E^*$ is relatively compact in $E$. Assume the following conditions:
\begin{itemize}
\item[(i)]There exists a neighborhood $B$ of some $u$ in $E$ and a constant $\sigma>0$ such that
$$J(v)\geq J(u)+\sigma \quad \text{for all }v\in\partial B.$$
\item[(ii)]Exists $w\not\in B$ such that $J(w)\leq J(u).$ 
\end{itemize}
Defining
$$\Gamma=\lbrace y \in C([0,1],E):\gamma(0)=u,\gamma(1)=w\rbrace$$
then there exists $u\in E$ such that $J^\prime(u)=0$ and $J(u)=c$, where
$$c=\inf_{\gamma\in\Gamma}\max_{0\leq t\leq 1}\lbrace J(\gamma(t)):t\in(0,1)\rbrace.$$
\end{theorem}

\begin{lemma}\label{35}
Assume that $\lbrace w_n\rbrace\subset W_{0}^{1,2}(\Omega)$ satisfies
\[\label{34}
J_{\epsilon,\lambda_n}(w_n)\leq C,\quad J_{\epsilon,\lambda_n}^{\prime}\rightarrow 0\text{ in }W_{0}^{-1,2}(\Omega),
\]
for $\lambda_n \rightarrow \lambda > 0.$ The sequence $(w_n)$ then admits a convergent subsequence in $W_{0}^{1,2}(\Omega).$
\end{lemma}

\begin{proof}
By \eqref{34} we have as $n\rightarrow +\infty$
\[
\int_{\Omega}e^{-a}|\nabla_gw_n|^2\, \mathrm{d} v_g-\lambda_n\int_{\Omega}e^{-a}g_\epsilon(w_n)w_n\, \mathrm{d} v_g=o(\Vert w_n\Vert_{W_{0}^{1,2}}).
\]
We have the inequality
$$\theta G_\epsilon(u)\leq ug_\epsilon(u) \quad \text{ for }u\geq M_\epsilon$$
for some $M_\epsilon > 0$ large and $\theta>2$. We obtain
\[
\begin{alignedat}{2}
C&\geq \frac{1}{2}\int_{\Omega}e^{-a}|\nabla_gw_n|^2\, \mathrm{d} v_g-\lambda_n\int_{\Omega}e^{-a}G_\epsilon(w_n)\, \mathrm{d} v_g\\
&=\left( \frac{1}{2}-\frac{1}{\theta} \right)\int_{\Omega}e^{-a}|\nabla_gw_n|^2\, \mathrm{d} v_g+\lambda_n\int_{\Omega}e^{-a}\left( \frac{1}{\theta}w_ng_\epsilon(w_n)-G_\epsilon(w_n) \right)\, \mathrm{d} v_g+o\left( \Vert w_n \Vert \right)\\
&\geq \left(  \frac{1}{2}-\frac{1}{\theta} \right)\int_{\Omega}e^{-a}|\nabla_gw_n|^2\, \mathrm{d} v_g+o\left( \Vert w_n \Vert_{W_{0}^{1,2}(\Omega)} \right)-C_\epsilon.
\end{alignedat}
\]
It follows that $\sup_{n\in\mathbb{N}}\Vert w_n \Vert_{W_{0}^{1,2}(\Omega)}<+\infty.$ We have the compactness of embedding $W_{0}^{1,2}(\Omega)\hookrightarrow L^{p+1}(\Omega)$ and thus, up to a subsequence, $w_n\rightharpoonup w$ weakly in $W_{0}^{1,2}(\Omega)$ and strongly in $L^{p+1}(\Omega)$ for some $w\in W_{0}^{1,2}(\Omega).$ It follows that
$$\int_{\Omega}e^{-a}|\nabla_gw|^2\, \mathrm{d} v_g=\lambda\int_{\Omega}g_\epsilon(w)w\, \mathrm{d} v_g$$
and we deduce that
\[
\begin{alignedat}{2}
\int_{\Omega}e^{-a}|\nabla_g(w_n-w)|^2\, \mathrm{d} v_g & =\int_{\Omega}e^{-a}|\nabla_g w_n|^2\, \mathrm{d} v_g-\int_{\Omega}e^{-a}|\nabla_g w|^2\, \mathrm{d} v_g+o(1)\\
&=\lambda_n\int_{\Omega}g_\epsilon(w_n)w_n\, \mathrm{d} v_g-\lambda\int_{\Omega}g_\epsilon(w)w\, \mathrm{d} v_g+o(1)\rightarrow 0.
\end{alignedat}
\]
as $n\rightarrow +\infty,$ and the lemma is proved.
\end{proof}

\begin{proof}[Proof of Theorem \ref{37}]
We first show that $J_{\epsilon,\lambda}$ has a mountain pass geometry in $W_{0}^{1,2}(\Omega)$. Since $\underline{u}_\lambda$ is a local minimum for $J_{\epsilon,\lambda}$ for $\lambda\nearrow\lambda^*$, condition $(i)$ of Theorem \ref{33} is satisfied. Consider $r>0$ such that $B_{2r}\subset\Omega$ and a cutoff function $\chi$ so that $\chi=1$ on $B_r$ and $\chi=0$ outside $B_{2r}$. Let $w_\epsilon=(1-\epsilon)\chi\in W_{0}^{1,2}(\Omega).$ In MEMS case, we have
\[
J_{\epsilon,\lambda}(w_\epsilon)\leq \frac{(1-\epsilon)^2}{2}\int_{\Omega}e^{-a}|\nabla\chi|^2\, \mathrm{d} v_g-\frac{\lambda}{\epsilon^2}\int_{B_r}e^{-a}\, \mathrm{d} v_g\rightarrow -\infty
\]
as $\epsilon\rightarrow 0$ and uniformly for $\lambda$ bounded away from $0$. With a similar argument we can prove the same result for Gelfand and Power cases. Thus we have
$$J_{\epsilon,\lambda}(\underline{u}_\lambda)\rightarrow J_{\epsilon,\lambda^*}(u_{\lambda^*}) \text{ as }\lambda\rightarrow\lambda^*$$
we get for $\epsilon>0$ sufficiently small that
$$J_{\epsilon,\lambda}(w_\epsilon)< J_{\epsilon,\lambda}(\underline{u}_\lambda)$$
holds for $\lambda$ close to $\lambda^*.$ It follows by Lemma \ref{35} that the functional $J_{\epsilon,\lambda}$ satisfies the Palais-Smale condition on $W_{0}^{1,2}(\Omega).$ We fix $\epsilon>0$ small enough and for $\lambda$ close to $\lambda^*$ we define
$$c_{\epsilon,\lambda}=\inf_{\gamma\in\Gamma}\max_{u\in\gamma}J_{\epsilon,\lambda}(u).$$
We can use the mountain pass theorem to get a solution $U_{\epsilon,\lambda}$ of \eqref{32} for $\lambda$ close to $\lambda^*.$ A similar proof as in Lemma \ref{31} shows that the convexity of $g_\epsilon$ ensures that problem \eqref{32} has a unique solution at $\lambda=\lambda^*,$ which is $u^*$. By elliptic regularity theory we get that $U_{\epsilon,\lambda}\rightarrow u^*$ uniformly in $C(\overline{\Omega})$. Thus $U_{\epsilon,\lambda}\leq t_0-\epsilon$ for $\lambda$ close to $\lambda^*.$ Therefore, $U_{\epsilon,\lambda}$ is a second solution for \eqref{01} bifurcating from $u^*,$ that we denote by $U_\lambda$. Since $U_\lambda$ is a mountain pass solution, $U_\lambda$ is not a minimal solution. Thus $U_\lambda$ is unstable solution of \eqref{01}.
\end{proof}

\end{document}